\documentclass[12pt]{article}

\usepackage[T1]{fontenc}
\usepackage[english]{babel}
\usepackage{amsfonts}
\usepackage{amssymb}
\usepackage{amsmath}
\usepackage{amsthm}
\usepackage{enumitem}

\usepackage{tikz}
\usetikzlibrary{decorations.pathreplacing,calligraphy}
\usetikzlibrary{decorations.markings}
\usetikzlibrary{patterns}
\usepackage{xfrac}
\usepackage{minitoc}

\usepackage{makecell}
\usepackage{graphicx}
\usepackage{caption} 
\usepackage{subcaption}
\usepackage{wallpaper}

\usepackage{dsfont}

\graphicspath{ {figures/} }

\usepackage{fancyhdr}
\usepackage{hyperref}
\usepackage{nomencl}
\usepackage{url}

\usepackage{cancel}
\usepackage{stmaryrd}

\usepackage{float}
\usepackage{autonum}

\usepackage[margin=2.5cm]{geometry} 

\usepackage{array}

\newlength{\oldparindent}
\theoremstyle{plain}
\newtheorem{theorem}{Theorem}[section]

\newtheorem{lemma}[theorem]{Lemma}
\newtheorem{prop}[theorem]{Proposition}

\theoremstyle{definition}
\newtheorem{definition}[theorem]{Definition}

\newtheorem{remark}[theorem]{Remark}

\newcommand{\sethyptag}[2]{\expandafter\gdef\csname hyptag@#1\endcsname{#2}}
\newcommand{\gethyptag}[1]{\csname hyptag@#1\endcsname}

\newcommand{\Hyp}[2]{%
  \par\vspace{0.7em}%
  \sethyptag{#2}{#1}%
  \hypertarget{hyp:#2}{}%
  \noindent\textbf{Hypothesis} {\normalfont(\gethyptag{#2}).}\hspace{0.5em}%
}

\newcommand{\hypref}[1]{\hyperlink{hyp:#1}{\textup{(\gethyptag{#1})}}} 

\newcommand{\vd}{\,\mathrm{d}}

\newcommand{\RR}{\mathbb{R}}
\newcommand{\PP}{\mathbb{P}}

\newcommand{\ind}[1]{\mathbf{1}_{#1}}

\newcommand{\sgn}{\operatorname{sgn}}

\usepackage[numbers]{natbib}
\title{Existence and uniqueness for \\
singular stochastic differential equations \\
with piecewise well-behaved coefficients} 

\author{  
Sara Mazzonetto\thanks{Université de Lorraine, CNRS, Inria, IECL, F-54000 Nancy, France, sara.mazzonetto@univ-lorraine.fr}
        \and
        Benoît Nieto\thanks{De Vinci Higher Education, De Vinci Research Center, Paris, France.
ESILV, 92916 Paris La Défense, France, benoit.nieto@devinci.fr}}            

\begin{document}

\maketitle

\noindent \textbf{Abstract.} 
We study existence and uniqueness for one-dimensional generalized stochastic differential equations with singular coefficients, including distributional drift and degenerate, possibly discontinuous, diffusion coefficients. 
Such singularities naturally encode changes in the dynamics at thresholds, including reflecting, skew, or sticky interface behavior.

We develop two directions. We provide sufficient conditions for pathwise uniqueness, under weak existence and uniqueness in law, without assuming uniform ellipticity or continuity of the diffusion coefficient. 
We also investigate a pasting approach for generalized stochastic differential equations that transfers strong existence and pathwise uniqueness, as well as weak existence and uniqueness in law, from local component equations to a global solution. To the best of our knowledge, this provides the first explicit pasting theorem yielding pathwise uniqueness in the setting of generalized stochastic differential equations.

As an application, we establish the first existence and uniqueness results for a class of skew sticky threshold Cox-Ingersoll-Ross-type diffusions, including the threshold Chan-Karolyi-Longstaff-Sanders process.


\medskip


\section{Introduction}

Stochastic differential equations (SDEs) driven by Brownian motion play a central role in many applications. 
A fundamental question is whether solutions exist and, if so, whether uniqueness in law or pathwise uniqueness holds. 
Early foundational work established existence and uniqueness results for various classes of SDEs, and subsequent research has strengthened these results, providing a solid basis for applications in areas such as quantitative finance, mathematical physics, and mathematical biology.

The case of SDEs with singular coefficients or distributional drift is considerably less explored, even in the one-dimensional setting. 
Nevertheless, such equations arise naturally when modeling diffusions that cross interfaces which locally perturb the dynamics by slowing them down, partially reflecting, or by altering them. 
Examples of such SDEs include the regime changing threshold diffusion models such as the threshold Cox-Ingersoll-Ross process (T-CIR)~\cite{decamps2006self,dong_wong_2015,zhang2024bond} (different CIR processes on different intervals), which can be seen as a special case of the threshold Chan-Karolyi-Longstaff-Sanders (T-CKLS) model~\cite{mazzonetto2024parameters}, as well as sticky-skew CIR-type processes~\cite{zhang2024hitting}.
The question of existence and uniqueness for such models is not covered by the classical SDE theory. 

In this paper, we investigate existence and uniqueness for one-dimensional generalized SDEs with distributional drift and measurable diffusion coefficients, allowing for both discontinuities and failure of uniform ellipticity. This class includes the models mentioned above. More precisely, we consider generalized SDEs of the form
\begin{equation}
    X_t
    = X_0
    + \int_{\mathbb{R}} L^{x}_t(X)\,\nu(\!\vd x)
    + \int_0^t b(X_s)\,\vd s
    + \int_0^t \sigma(X_s)\,\vd B_s,
    \qquad t \geq 0,
    \label{eqn:SDE_general}
\end{equation}
where $B$ is a one-dimensional Brownian motion, $b$ and $\sigma$ are measurable functions, $\nu$ is a signed measure singular with respect to the Lebesgue measure, and $L^x_t(X)$ denotes the symmetric local time of the process.
The coefficients $b$ and $\sigma$ are allowed to be \emph{discontinuous}, and the diffusion coefficient $\sigma$ is not assumed to satisfy a uniform \emph{ellipticity condition}, that is, $\sigma$ is not bounded from below by a strictly positive constant, as for CIR-type models.
The distributional drift term includes skew diffusions as a special case, corresponding to $\nu(\!\vd x)=\beta\,\delta_\theta(\!\vd x)$ for some $\theta\in\mathbb{R}$ and $\beta\in(-1,1)$. The case $\beta=1$ or $\beta=-1$ corresponds to positive and negative reflection.

Among the works that have addressed existence and uniqueness questions for SDEs with irregular coefficients, let us mention~\cite{bass_chen_05,bass1987uniqueness,bleiengelbert2014, engel_schm_89_3,le2006one}. These works allow for distributional drift $\nu\not\equiv0$ or possibly discontinuous (finite-jump) diffusion coefficients. 
In all these contributions, except for~\cite{bleiengelbert2014,engel_schm_89_3}, a uniform ellipticity condition on the diffusion coefficient is assumed.
In the latter references the diffusion coefficients may have vanish at some $\theta\in \mathbb{R}$ (i.e.~$\sigma(\theta)=0$), possibly leading to stickiness phenomena. Nevertheless, these results do not apply to the T-CIR model, because they provide existence and uniqueness result up to the first time the process hits $0$.

One of the motivations of our work is to provide a general framework for SDEs whose coefficients are \emph{piecewise well-behaved} but the diffusion coefficient is not uniformly elliptic. 
By \emph{piecewise well-behaved}, we mean that the coefficients coincide locally with those of other SDEs which satisfy existing (for instance, classical) existence and uniqueness results. 
Although pasting or piecing-out (construct a Markov process from killed ones)  techniques for Markov processes have been studied for a long time (see~\cite{ikeda1966construction,kopytko2009problem,kopytko2006analytical,lamberton2013optimal,meyer1975renaissance,mota2014continuous,nagasawa1976note,Werner_pasting}), we are not aware of works explicitly constructing a strong solution via pasting strong solutions of singular and non-singular SDEs and proving pathwise uniqueness through such techniques. 

Note that strong existence and uniqueness for two SDEs do not automatically guarantee global strong existence or pathwise uniqueness for an SDE whose dynamics are given by one equation on a semiaxis and by another on the complementary semiaxis. 
A famous counterexample is given by the sign-SDE ($b\equiv0$ and $\sigma=\mathds{1}_{(0,+\infty)}-\mathds{1}_{(-\infty,0]}$), for which existence and uniqueness in law hold, while pathwise uniqueness fails. On $(0,+\infty)$, the coefficients coincide with the ones of an SDE satisfying strong existence and uniqueness and the solution is a Brownian motion. The same holds on $(-\infty,0)$.
However, it is natural to expect that, if the coefficients are sufficiently well-behaved in a neighborhood of the thresholds, then existence and uniqueness should be transferred.

A first contribution of this paper is to provide sufficient conditions, a mathematical justification, and a ready to use formulation of a pasting approach for (generalized) SDEs, which is particularly well suited to equations with piecewise-defined coefficients. Starting from two auxiliary processes that solve SDEs of the same form as \eqref{eqn:SDE_general} on their respective regimes, and for which existence and uniqueness are known to hold, we construct a global process by pasting these local solutions together. We then provide sufficient conditions under which this pasted process solves the original SDE~\eqref{eqn:SDE_general} and inherits the corresponding well-posedness properties. 

These results are established in Proposition~\ref{prop:weak_pasting} and Theorem~\ref{thm:pasting}, which give natural sufficient compatibility conditions on the coefficients of the auxiliary SDEs for, respectively, weak existence and uniqueness in law, and strong existence and pathwise uniqueness. 
Although the compatibility conditions coincide, the proof of Theorem~\ref{thm:pasting} does not rely on Proposition~\ref{prop:weak_pasting}. 
In the weak setting, the use of pasting arguments is not new in the construction of Markov processes.  
By contrast, the corresponding formulation of the pasting result for strong existence and pathwise uniqueness is new.
Moreover, we show that the compatibility conditions formulated for the strong setting also imply weak existence and uniqueness in law in the context of SDEs and allow for a self-contained proof in the weak setting as well. 

Independently of the pasting construction, we also develop a general criterion for pathwise uniqueness in situations where weak existence and uniqueness in law are known to hold. 
This result is stated in Theorem~\ref{Theorem_uniqueness_thresh}, and leads to a comparison theorem given in Theorem~\ref{Theorem_comparison}.
The statement and the proof are inspired by the seminal work~\cite{le2006one}, but we extend the framework beyond the uniformly elliptic case to include diffusion coefficients that may be neither uniformly elliptic nor continuous. 

To summarize, a key feature of all our results is that we do not require global regularity assumptions on the coefficients of the SDEs. In particular, the diffusion coefficient does not need to be uniformly elliptic or continuous, not even across different regimes. 
Our results therefore apply to SDEs with possibly degenerate and discontinuous diffusion coefficients. In the existing literature, results of this kind are available either for discontinuous (uniformly elliptic) coefficients or for degenerate coefficients, except if the solutions are considered up to some stopping time. 
In conclusion, our results provide flexible and robust tools for establishing strong well-posedness for SDEs with singular coefficients.

As an application to threshold-type SDEs, in Section~\ref{sec:application}, we provide  the first weak and strong existence and uniqueness results for the SDE satisfied by the T-CKLS process,
\begin{equation}
X_t = X_0
+ \int_0^t \bigl(a(X_s)-b(X_s)X_s\bigr)\,\vd s
+ \int_0^t \sigma(X_s)|X_s|^{\gamma(X_s)}\,\vd B_s,
\qquad t\ge0,
\label{eq:TCKLS}
\end{equation}
where $X_0>0$ is independent of the Brownian motion $(B_t)_{t\ge0}$, and where the coefficients $a$, $b$, $\sigma$, and $\gamma$ are piecewise constant and may be discontinuous at finitely or countably many thresholds. We assume that $\gamma(x)\in[1/2,1]$ in a neighborhood of $0$.

The T-CIR process is a special case of this model, with $\gamma\equiv 1/2$, and we provide the first existence and uniqueness result for this model as well. 
See Section~\ref{sec:application} for more details and extensions of this model for which the interfaces do not only exhibit threshold effects (coefficient jumps) but also stickiness and skewness.

\medskip

\textbf{Outline.} 
Section~\ref{sec:main_results} is devoted to existence and uniqueness results via pasting technique or local time related techniques. In a dedicated subsection, we comment on the results and some extensions  in relation to existing literature. 
The proof of the main theorems is presented in Appendix~\ref{sec:app:proofs}, after  Appendix~\ref{sec:app:classics}, where we recall definitions and classical results on existence and uniqueness, drift removal and non-explosion. 
In Section~\ref{sec:application}, we apply the pasting technique to specific frameworks of singular diffusions, and in particular to the T-CKLS model. 
Finally, in Appendix~\ref{sec:app:CKLS}, we illustrate the other approach in the case of T-CKLS while discussing some properties of this process. 

\medskip

\textbf{Notation.}
In this paper, any filtered probability space $(\Omega, \mathcal{F}, (\mathcal{F}_t)_{t\geq 0}, \mathbb{P})$ is supposed to satisfy the usual conditions. 
Moreover, we follow the definitions of weak solution, strong solution, uniqueness in law, pathwise uniqueness provided in~\cite[Definitions 5.2.1, 5.3.1, 5.3.2, and 5.3.4]{ks}. We recall them in Appendix~\ref{sec:app:classics}. 
Note that, without loss of generality, one can consider the initial condition of the SDEs, $X_0$, to be deterministic. For extensions to non-deterministic $X_0$, we refer to~\cite[Proposition IX.1.4]{revuz2013continuous}, for instance, to cover the case in which the non deterministic $X_0$ is independent of the driving  $(\mathcal{F}_t)_{t\geq 0}$-Brownian motion and it is also $\mathcal{F}_t$-measurable for all $t\geq 0$.

Throughout this work, $L_t^x(X)$ will denote the symmetric local time of the semimartingale $X$ cumulated at the point $x$  until time time $t$, and $\ell_t^x(X)$ is the right local time. 
In both cases, we work with the standard version that is continuous in $t$ and càdlàg in $x$.
Then, the right local time satisfies $\mathbb{P}$-a.s. 
\begin{equation}
    \ell^x_t(X)=\lim_{\varepsilon \to 0^+} \frac1\varepsilon \int_0^t \ind{(x,x+\varepsilon)}(X_s) \vd \langle X, X\rangle_s
\end{equation}
and the symmetric local time satisfies $\mathbb{P}$-a.s. 
\begin{equation}
    L^x_t(X)=\lim_{\varepsilon \to 0^+} \frac1{2\varepsilon} \int_0^t \ind{(x-\varepsilon,x+\varepsilon)}(X_s) \vd \langle X, X\rangle_s.
\end{equation}

We also use the convention $1/0=\infty$ and $1/\infty=0$.

\section{Main results}
\label{sec:main_results}

In this section, we present weak and strong existence and uniqueness results for generalized SDE \eqref{eqn:SDE_general}. Broadly speaking, the coefficients $\nu$, $b$, and $\sigma$ are locally well-behaved on each interval. 
In particular, we allow them to be piecewise-defined (and discontinuous).
The measure $\nu$ satisfies the following; further details on the measurable functions $b$ and $\sigma$ will be given later.

\Hyp{$\mathcal{NU}$}{hyp:nu}
The measure $\nu$ is a locally finite signed measure on $\mathbb R$, singular with respect to
Lebesgue measure, and satisfies $|\nu(\{x\})|< 1$ for all $x\in\mathbb R$.

Let us note that, this setting includes skew diffusions: $\theta\in \mathbb{R}$ is a skew interface if $\nu(\{\theta\})=\beta \in (-1,1)\setminus \{0\}$ (it would be reflected if $|\nu(\{\theta\})|=1$). If we consider the right local time instead of the symmetric one, $\theta$ would be a skew interface if $\nu(\{\theta\})=\beta \in (-\infty,1/2)\setminus \{0\}$ (a positive reflection at $\theta$ corresponds to $\nu(\{\theta\})=1/2$). 
Moreover, the term $\int_0^t b(Y_s)\vd s$ can be rewritten as $\int_\mathbb{R}L_t^x(X) \nu(\!\vd x)$ for a signed measure $\nu_{b,\sigma}$ which is absolutely continuous with respect to the Lebesgue measure.

In this section we present two approaches for existence and uniqueness.
\begin{itemize}[leftmargin=*]
\item
We first develop a \emph{pasting} approach for weak existence, uniqueness in law, strong existence, and pathwise uniqueness. 
It consists in combining compatible local solutions across different intervals. 

For the sake of completeness, we provide statement and proofs for all results: weak and strong existence and uniqueness. Clearly, for any specific model, there is no need to consider all these results. 
Indeed, there are connections between these notions, such as strong existence implies weak existence, or Yamada-Watanabe~\cite{yamada1971uniqueness} result or Cherny's~\cite[Theorem 3.2]{cherny_2002}. 

Nevertheless, since for some SDEs (e.g. sticky ones) there cannot be a strong solution nor pathwise uniqueness, it is worth mentioning separate results for weak existence and uniqueness in law, which can be applied to sticky diffusions.

\item
Next, in Section~\ref{ssec:pu_leg}, we present a pathwise uniqueness result applicable to SDEs for which weak existence and uniqueness in law is known.
This result can be seen as a generalization of Le Gall~\cite[Theorem 1.3]{le2006one} to the non-uniformly elliptic case and it implies, in this setting, a new comparison result which generalizes~\cite[Theorem~3.4]{le2006one} and \cite[Theorem~3.1]{bass_chen_05} when the two SDEs considered have the same generalized drift, i.e.~the same $\nu$. 
\end{itemize}

Finally, we propose some discussion of the results and describe the connection between them and with the existing literature. The proofs of our results are given in Appendix~\ref{sec:app:proofs}.

Note that the existence and uniqueness results we propose are based on knowing at least existence and uniqueness on some SDEs. We recall some well known general results in Appendix~\ref{sec:app:classics}.

\subsection{Pasting approach} \label{ssec:pasting}

Many classical or more recent references deal with pasting or piecing-out of diffusions and Markov processes~\cite{ikeda1966construction,kopytko2009problem,kopytko2006analytical,meyer1975renaissance,mota2014continuous,nagasawa1976note,Werner_pasting}. 
This method has been useful to define new processes such as the snapping out Brownian motion~\cite{lejay2016snapping}. 
We now provide a version suited for SDEs, under a compatibility condition for the coefficients of the different SDEs we are pasting together, which allows to provide, both weak existence and uniqueness in law and the stronger results on strong existence and pathwise uniqueness.

In this section, we assume that there exists an interface $\vartheta\in \mathbb{R}$ which separates the dynamics above and below $\vartheta$, in the sense that the dynamics on the two sides of $\vartheta$ are described by two processes satisfying an SDEs such as~\eqref{eqn:SDE_general} with suitable coefficients $(\nu_-,\mathfrak{b}_-,\mathfrak{G}_-)$ and $(\nu_+,\mathfrak{b}_+,\mathfrak{G}_+)$, respectively, and where both descriptions remain valid on a neighborhood of $\vartheta$.

Let us be more precise.
Let us consider two SDEs
\begin{equation}
    \label{eqn:SDE_1}
    X_t = x_0^+ + \int_{\mathbb{R}} L^x_t(X)\,\nu_+(\!\vd x)
    + \int_0^t \mathfrak{b}_+(X_s)\, \mathrm{d}s
    + \int_0^t \mathfrak{G}_+(X_s)\, \mathrm{d}B_s,
\end{equation}
and
\begin{equation}
    \label{eqn:SDE_2}
    X_t = x_0^- + \int_{\mathbb{R}} L^x_t(X)\,\nu_-(\!\vd x)
    + \int_0^t \mathfrak{b}_-(X_s)\, \mathrm{d}s
    + \int_0^t \mathfrak{G}_-(X_s)\, \mathrm{d}B_s.
\end{equation}

\Hyp{$\mathcal{WEU}$}{hyp:weak_pasting}
Assume that weak existence and uniqueness in law hold for both SDE~\eqref{eqn:SDE_1} and~\eqref{eqn:SDE_2}. 
Moreover, assume that the solutions are non-explosive. 

We denote by $I_+$ and $I_-$ the corresponding state spaces of the solutions to SDE~\eqref{eqn:SDE_1} and~\eqref{eqn:SDE_2}.
By \emph{state space} $I$ we mean the smallest set such that $\mathbb P(X_t\in I \text{ for all } t\ge 0)=1$.
Note that the initial condition $x_0^\pm\in \text{Interior}(I_\pm)$.

\Hyp{$\mathcal{COMP}$}{hyp:pasting}
Assume that the intersection of the state spaces $I_-\cap I_+$ is non-empty and it contains an interval $(\vartheta-\varepsilon,\vartheta+\varepsilon)$ for some $\vartheta\in \mathbb{R}$ and $\varepsilon>0$ such that
\begin{equation}\label{eq:overlap_coeffs}
\mathfrak{b}_+(x)=\mathfrak{b}_-(x),\qquad \mathfrak{G}_+(x)=\mathfrak{G}_-(x),
\qquad \text{ for all } x\in(\vartheta-\varepsilon,\vartheta+\varepsilon),
\end{equation}
and that the measures coincide on the overlap region, in the sense that
\begin{equation}\label{eq:overlap_nu}
\nu_+|_{(\vartheta-\varepsilon,\vartheta+\varepsilon)}
=
\nu_-|_{(\vartheta-\varepsilon,\vartheta+\varepsilon)},
\end{equation}
i.e.\ $\nu_+(A)=\nu_-(A)$ for every Borel set $A\subseteq(\vartheta-\varepsilon,\vartheta+\varepsilon)$.

Next, let 
\begin{equation} \label{eq:coeffs}
b(x):=
\begin{cases}
\mathfrak{b}_+(x), & x\ge \vartheta,\\[2pt]
\mathfrak{b}_-(x), & x<\vartheta,
\end{cases}
\qquad
\sigma(x):=
\begin{cases}
\mathfrak{G}_+(x), & x\ge \vartheta,\\[2pt]
\mathfrak{G}_-(x), & x<\vartheta,
\end{cases}
\end{equation}
and the local-time drift is defined by
\begin{equation}\label{eq:nu_piecewise}
\nu(A):=\nu_+(A\cap[\vartheta,\infty))+\nu_-(A\cap(-\infty,\vartheta)),
\qquad A\in\mathcal B(\mathbb R).
\end{equation}
(Any fixed convention at $\vartheta$ is acceptable; we adopt the right-hand convention here.)

With $b$ and $\sigma$ as in~\eqref{eq:coeffs}, the generalized SDE~\eqref{eqn:SDE_general} is precisely the one associated with the two-regime dynamics~\eqref{eqn:SDE_1}-\eqref{eqn:SDE_2} separated by the interface/threshold $\vartheta$. Finally, we set
\begin{equation} \label{eq:state_space}
    I := (I_+ \cap (\vartheta-\varepsilon,+\infty)) \cup (I_- \cap (-\infty,\vartheta+\varepsilon)).
\end{equation}


\begin{prop}[Weak existence and uniqueness in law by piecing-out]
\label{prop:weak_pasting}
Assume Hypothesis~\hypref{hyp:weak_pasting} and Hypothesis~\hypref{hyp:pasting} and let $I$ given by~\eqref{eq:state_space}. 
Then for every $x\in \text{Interior}(I)$ there exists a unique weak solution $(X,B)$ to~\eqref{eqn:SDE_general}
with $X_0=x$ and $(\nu,b,\sigma)$ defined as in~\eqref{eq:coeffs}.
Moreover, this solution is non-explosive and $I$ is its state space, i.e. 
$
    \PP\big(X_t\in I \ \forall t\ge 0\big)=1.
$
\end{prop}

We now aim to construct a strong solution $X$ to the general SDE~\eqref{eqn:SDE_general} 
by pasting together strong solutions corresponding to different regimes of the coefficients. 
As we recalled earlier in this paper, sometimes one can deduce strong existence and uniqueness from weak existence and uniqueness and additional properties. We do so in the next section. Instead, this is not the strategy of the next theorem, which is independent from the previous one on weak existence and uniqueness.

\Hyp{$\mathcal{SEU}$}{hypothesis_strong}
Assume that strong existence and pathwise uniqueness hold for the SDEs~\eqref{eqn:SDE_1}-\eqref{eqn:SDE_2}, that the solutions are non-explosive, and that Hypothesis~\hypref{hyp:pasting} holds. 

\begin{theorem}
\label{thm:pasting}
Assume Hypothesis~\hypref{hypothesis_strong}, let $I$ as in~\eqref{eq:state_space} and $(\nu,b,\sigma)$ defined as in~\eqref{eq:coeffs}.
Then, for every $x$ in the interior of $I$, there exists a unique strong solution to the SDE~\eqref{eqn:SDE_general} with $X_0=x$. Moreover, this solution is non-explosive and its state space is $I$. 
\end{theorem}

\subsection{Pathwise uniqueness when weak existence and uniqueness hold} \label{ssec:pu_leg}

In this section, we provide a pathwise uniqueness result for SDE \eqref{eqn:SDE_general}, once provided that weak existence and uniqueness in law hold. This extends the existing literature to include some threshold diffusions for which ellipticity condition of the diffusion coefficient is not satisfied. We also propose a comparison theorem based on this result. 

Let us first introduce some notation.  
For any measurable function $\sigma \colon \mathbb{R} \to \mathbb{R}$, let
\begin{equation} \label{eq:esigma}
    E_\sigma :=\{x\in \mathbb{R} \colon \int_{\mathcal{U}(x)}{(\sigma(y))^{-2}} \vd y=+\infty \ \forall \mathcal{U}(x) \text{ open set containing }x\}.
\end{equation}
For any $U\subseteq \mathbb{R}$ and $\varepsilon>0$ let $U^{\varepsilon}:=\{x\in \mathbb R \colon \exists y \in U \text{ such that } |x-y| < \varepsilon \}$ the enlarged set. Note that if $U=\emptyset$ then $U_\varepsilon=\emptyset$.

\Hyp{$\mathcal{COEF}$}{hyp:engel-schm:modif}
Assume that $b$ and $\sigma$ are measurable, locally bounded functions, $\nu$ satisfies assumption~\hypref{hyp:nu}, and there exists  
an open set $U \subseteq \mathbb R$ such that $E_\sigma \subseteq U$. 
Moreover, assume that for every compact $K \subseteq \mathbb R\setminus E_\sigma$, 
there exist 
\begin{itemize}[nolistsep]
\item a strictly positive $\varepsilon$, 
\item a measurable function $\rho \colon \mathbb R\to [0,+\infty]$, satisfying $\int_{\mathcal{V}(0)} 1/{\rho(y)} \vd y=\infty$, for any neighborhood $\mathcal{V}(0)$ of $0$, 
\item and an increasing function $f \colon K \to \mathbb{R}$
\end{itemize}
such that 
\begin{enumerate}[label=$(\mathcal{C}_\arabic*)$]
    \item \label{item:hyp:1}
     for all $x,y\in U^\varepsilon$ such that $|x-y| <\varepsilon$: 
    \begin{equation}
        (\sigma(x)-\sigma(y))^2 \leq \rho(x-y);
    \end{equation} 
    \item \label{item:hyp:2}
    $\inf_{x\in K} \sigma(x) >0$ 
    and for all $x,y\in K$ such that $|x-y| <\varepsilon$: 
    \begin{equation}
        (\sigma(x)-\sigma(y))^2 \leq \frac{\rho(x-y) |f(x)-f(y)|}{|x-y|}.
    \end{equation}
\end{enumerate}

\begin{remark}
    If $E_\sigma=\emptyset$, then $U$ can be chosen to be the empty set as well and 
    assumption~\ref{item:hyp:1} trivially holds. 
\end{remark}
\Hyp{$\mathcal{SOL}$}{hyp:pathwise}
Assume that there exists a weak solution to~\eqref{eqn:SDE_general} and that every weak solution to~\eqref{eqn:SDE_general} is non-explosive. 
\begin{prop}
Assume that $\nu \equiv 0$ and that Hypothesis~\hypref{hyp:engel-schm:modif} and  Hypothesis~\hypref{hyp:pathwise} hold.  
Let $X^{(1)}$ and $X^{(2)}$ be two weak solutions of~\eqref{eqn:SDE_general}, defined with
respect to the same Brownian motion on the same probability space, and such that
$X^{(1)}_0 = X^{(2)}_0$ $\mathbb{P}$-a.s.  
Then, the right local time 
$
\ell_\cdot^0(X^{(1)} - X^{(2)}) = 0$ $\mathbb{P}\text{-a.s.}
$.
\label{prop_uniqueness_thresh}
\end{prop}
\begin{theorem}[Pathwise uniqueness]
Assume that Hypothesis~\hypref{hyp:engel-schm:modif} and \hypref{hyp:pathwise} hold, and that uniqueness in law holds for~\eqref{eqn:SDE_general}.  
Then pathwise uniqueness holds. 
By Yamada-Watanabe~\cite{yamada1971uniqueness} there exists a unique strong solution to
\eqref{eqn:SDE_general}.
\label{Theorem_uniqueness_thresh}
\end{theorem}

\begin{theorem}[Comparison theorem for thresholds SDEs] \label{Theorem_comparison}
Let $X^{(1)}, X^{(2)}$ be two non-explosive weak solutions to the following SDEs with respect to the same Brownian motion $B$, on the same probability space: for $i\in \{1,2\}$
\begin{equation}
X^{(i)}_t=X_0^{(i)}+ \int_{\mathbb{R}} L^x_t(X^{(i)}) \nu(\!\vd x)+ \int_0^t b_i(X^{(i)}_s)\vd s+\int_0^t \sigma(X^{(i)}_s)\vd B_s,\quad \forall t\geq 0
\end{equation}
where $b_i$ and $\nu$ and $\sigma$ satisfy Hypothesis~\hypref{hyp:engel-schm:modif}. 
Assume further that 
for all $x\in \mathbb{R}$ $b_1(x)\geq b_2(x)$, one of two functions $b_1$ or $b_2$ is Lipschitz, and $X^{(1)}_0\geq X^{(2)}_0$ $\mathbb{P}$-a.s.. 
Then it holds $\mathbb{P}$-a.s.~for all $t\geq 0$ that 
$X^{(1)}_t\geq X^{(2)}_t$.
\end{theorem}


\begin{remark}[Weakening the assumptions of Theorem~\ref{Theorem_comparison}]
Note that the assumption on Lipschitzianity in Theorem~\ref{Theorem_comparison} can be relaxed.
    More precisely, on one side (globally) Lipschitz can be replaced by locally Lipschitz.
    On another side, if both $b_1$ and $b_2$ are continuous, one can find a Lipschitz function $\tilde{b}$ such that $b_1 < \tilde{b} < b_2$ and apply the Theorem~\ref{Theorem_comparison} twice: to $b_1$ and to $b_2$ with $\tilde{b}$. 
\end{remark}

\subsection{Comments and extensions}
\label{sec:comments}
\paragraph{Pasting approach.}
The formulation of our pasting approach is rather flexible.
Let us first note that the coefficients $\mathfrak{b}_-,\mathfrak{b}_+,\mathfrak{G}_-,\mathfrak{G}_+$ of SDEs \eqref{eqn:SDE_1}-\eqref{eqn:SDE_2} are not assumed to be continuous, actually they may admit infinitely many discontinuities.
The choice of the auxiliary SDEs~\eqref{eqn:SDE_1}-\eqref{eqn:SDE_2} is not unique, and there is no intrinsic choice for the threshold $\vartheta$ and compatibility interval $(\vartheta-\varepsilon, \vartheta+\varepsilon)$ which are not unique. 
Moreover, the statement can be easily reformulated, if needed, with finitely many thresholds together with the respective compatibility conditions of the different dynamics.
Since no openness of $I$, $I_+$, or $I_-$ is used, our setting allows for boundary behaviors such as reflection and absorption, which are already part of the auxiliary SDEs~\eqref{eqn:SDE_1}-\eqref{eqn:SDE_2}. 
Furthermore, Hypothesis~\hypref{hyp:weak_pasting} can be adapted to prove the analogue of Proposition~\ref{prop:weak_pasting} in the case of sticky SDEs
\begin{equation}
    \begin{cases}
    X_t = X_0 + \int_{\mathbb{R}\setminus\{\theta\}} L^x_t(X)\,\nu(\!\vd x)
    + \int_0^t b(X_s) \ind{\{X_s\neq \theta\}} \, \mathrm{d}s + \int_0^t \sigma(X_s) \ind{\{X_s\neq \theta\}}\, \mathrm{d}B_s,\\    
    \int_0^t \ind{\{X_s=\theta\}} \vd s = \kappa L^{\theta}_t(X), 
    \end{cases}
    \label{eqn:SDE_sticky}
\end{equation}
where $\kappa\geq 0$, and $\theta\in \mathbb{R}$.
The only difference is that the analogous of the auxiliary SDEs~\eqref{eqn:SDE_1}-\eqref{eqn:SDE_2} which satisfy weak existence and uniqueness are possibly systems of SDEs. The compatibility condition can stay the same if, for instance the compatibility interval $(\vartheta-\varepsilon,\vartheta+\varepsilon)$ does not intersect $\theta$. We only discuss extension of Proposition~\ref{prop:weak_pasting} because stickiness implies lack of strong solution/path uniqueness \cite{engelbert_peskir_2014}.
There is no strong existence for~\eqref{eqn:SDE_sticky}, see e.g.~\cite{engelbert_peskir_2014}.  

Both strong and weak existence and uniqueness results obtained via pasting of SDEs can be readily adapted to the case of reflected SDEs, corresponding to allowing $\nu$ to take the values $1$ and $-1$. It suffices to consider reflected auxiliary SDEs, and the compatibility condition remains unchanged.

Our results, via the pasting approach or the other approach, require additional existence and uniqueness results.
Some existing weak existence and uniqueness, results for non sticky SDEs are provided in Appendix~\ref{sec:app:classics}. For (weak)  existence and uniqueness of sticky diffusions (uniformly elliptic diffusion outside the sticky interfaces) we refer to~\cite{salins2017markov} and for (strong) existence and uniqueness~\cite{bleiengelbert2014,engel_schm_89_3, MG_S_Y_25} (the latter for continuous diffusion coefficient), and \cite[Theorem IX.3.8]{revuz2013continuous}.

\paragraph{Other approach.} 
Theorem~\ref{Theorem_uniqueness_thresh} is similar in spirit to that of Le Gall~\cite{le2006one} and  Hypothesis~\hypref{hyp:engel-schm:modif} is inspired by Engelbert and Schmidt, in particular by Theorem~4.48 in~\cite{engel_schm_89_3}. But in that result in~\cite{engel_schm_89_3}, as well as in the results in~\cite{bleiengelbert2014}, the solutions are typically stopped upon hitting $E_\sigma$, here instead we deal with non-stopped solutions. 
Indeed, Hypothesis~\ref{item:hyp:1} controls the behavior of the diffusion coefficient $\sigma$ near the set $E_\sigma$ ensuring enough regularity.
Hypothesis~\ref{item:hyp:2} is more classical. It guarantees local ellipticity and allows for \emph{gentle discontinuities} (such as finite jumps) in regions where $\sigma$ stays away from zero. 
It does not require Lipschitz continuity and is sufficient to ensure uniqueness in law away from $E_\sigma$, similarly to conditions found in \cite{engel_schm_89_3}. 
Note that a recent interesting article~\cite{menoukeu2019pathwise} contains a result in the same spirit as ours but with different assumptions which, on the one hand, exclude discontinuous diffusion coefficients and general CKLS shape, and, on the other hand, allows for time inhomogeneous coefficients.

As an alternative to Hypothesis~\hypref{hyp:engel-schm:modif}, we next introduce a set of assumptions $(\mathcal{T})$, that are stronger but more explicitly designed for threshold diffusions, such as the T-CKLS model~\eqref{eq:TCKLS}. 
In threshold diffusions which are solutions to SDEs, the coefficients are typically piecewise smooth or piecewise regular. More precisely, we consider a threshold SDE with threshold $\theta$ (the diffusion coefficient $\sigma$ is discontinuous at $\theta$). We have the following assumption.
\Hyp{$\mathcal{T}$}{hyp:cond:bsigma}
Let $b$ and $\sigma$ be measurable and locally bounded functions, and $\nu\equiv 0$. 
Assume that there exists $\varepsilon > 0$ such that:

\begin{enumerate}[label=\upshape $(\mathcal{T}_\arabic*)$]

    \item \label{item_ellip:th_uniqueness} 
    \emph{(Ellipticity near the threshold)}  
    There exists $\alpha > 0$ such that for all $x \in (\theta - \varepsilon, \theta + \varepsilon)$ it holds 
    $\sigma(x) \geq \alpha$.

    \item \label{item_disc:th_uniqueness}
    \emph{(Controlled jump at the threshold)}  
    There exist $f \colon [\theta - \varepsilon, \theta + \varepsilon] \to \mathbb{R}$ is a bounded increasing function such that for all $x, y \in [\theta - \varepsilon, \theta + \varepsilon]$ such that $(x - \theta)(y - \theta) \le 0$, it holds that
    \begin{equation}
        |\sigma(x) - \sigma(y)|^2 \leq |f(x) - f(y)|.
    \end{equation}

    \item \label{item_classic:th_uniqueness}
    \emph{(Regularity inside each region)}  
    Let $I_1 := (-\infty, \theta)$ and $I_2 := [\theta, +\infty)$.  
    For each region $I \in \{I_1, I_2\}$, one of the following holds:

    \begin{enumerate}[label=\upshape $(\mathcal{T}_{3,\arabic*})$]

        \item \label{item_classic_cont:th_uniqueness}
        \emph{(Continuous case)}  
        There exists a function $\rho \colon (0, \infty) \to (0, \infty)$ such that $\rho(h) \ge h$ for all $h > 0$, 
        $
            \int_{\mathcal{V}(0)} 1/{\rho(h)}\,\mathrm{d}h = \infty
        $ 
        for any neighborhood $\mathcal{V}(0)$ of $0$ and there exists for all $x, y \in I$
        with $|x - y| < \varepsilon$:
        \begin{equation}
            |\sigma(x) - \sigma(y)|^2 \leq \rho(|x - y|).
        \end{equation}

        \item \label{item_classic_disc:th_uniqueness}
        \emph{(Discontinuous case: controlled jumps and ellipticity)}  
        $\sigma|_{I} \ge \alpha$ for some constant $\alpha > 0$, and there exists for all $x, y \in I$ with $|x - y| < \varepsilon$:
        \begin{equation}
            |\sigma(x) - \sigma(y)|^2 \leq |g(x) - g(y)|,
        \end{equation}
        where $g \colon I \to \mathbb{R}$ is an increasing, locally bounded function.

    \end{enumerate}
\end{enumerate}

\begin{remark}
    If \ref{item_ellip:th_uniqueness}-\ref{item_disc:th_uniqueness} and  \ref{item_classic_disc:th_uniqueness} hold for both $I_1$ and $I_2$, then Hypothesis~\hypref{hyp:cond:bsigma} is equivalent to the assumptions in~\cite{le2006one}: 
    $\sigma$ is uniformly elliptic (bounded from below by a strictly positive constant) and there exists $f$ increasing locally bounded such that for all $x,y\in \mathbb{R}$, $(\sigma(x)-\sigma(y))^2 \leq |f(x)-f(y)|$. 
\end{remark}

\begin{remark}
    As emphasized in \cite{le2006one} and \cite[Theorem 4.41]{engel_schm_89_3}, the assumption \ref{item_classic_cont:th_uniqueness} can be relaxed using localized control involving integrability conditions: there exists a function $\rho \colon \mathbb{R} \to (0,\infty)$ such that 
    \begin{equation} \label{eqn:local_0_1}
        |\sigma(x)-\sigma(y)|^2\leq  a(x)\rho(x-y),
    \end{equation} 
    where $a \sigma^{-2}\in L^1_{\text{loc}}(\mathbb{R}\setminus  E_\sigma)$.
\end{remark}

To conclude, let us mention that our comparison theorem, Theorem~\ref{Theorem_comparison} generalizes~\cite[Theorem~3.4]{le2006one} and \cite[Theorem~3.1]{bass_chen_05} to the case of non-elliptic diffusion coefficient but we can only deal with the same signed measure $\nu$ for the two equations. Indeed, the proof strategy of the just cited references seems to fail in our case, due to the lack of ellipticity of the diffusion coefficient.  

\section{Applications to some specific models}
\label{sec:application}

In this section, we apply the general pasting results of this paper to a class of one-dimensional diffusions that are solutions to SDEs with singular coefficients, focusing on a skew-sticky extension of the T-CKLS model. 
For the pure T-CKLS model, for completeness, in Appendix~\ref{sec:app:CKLS}, we present the other approach proposed in this paper. 

We have discussed in Section \ref{sec:main_results} the extension of Proposition~\ref{prop:weak_pasting} to sticky generalized SDEs.

While our results apply to coefficients with a finite or countable number of threshold discontinuities, we restrict ourselves here, without loss of generality, to the case of a single threshold. 

We now introduce the skew-sticky T-CKLS model.
\begin{equation}
\left\{
\begin{array}{ll}
X_t = x_0 + \displaystyle \int_0^t \bigl(a(X_s) - b(X_s) X_s\bigr)\,
        \mathds{1}_{\{X_s \neq \theta\}} \,\mathrm{d}s
    +  \displaystyle \int_0^t \sigma(X_s) \left|X_s\right|^{\gamma(X_s)}\,
        \mathds{1}_{\{X_s \neq \theta\}}\,\mathrm{d}B_s
    + \beta\, L_t^{\theta}(X), \\[6pt]
\displaystyle \int_0^t \mathds{1}_{\{X_s \neq \theta\}}\,\mathrm{d}s
    = \kappa\, L_t^{\theta}(X),
\end{array}
\right.
\label{eq:sticky_skew_CKLS}
\end{equation}
with initial condition $x_0>0$, interface point $\theta>0$, sticky parameter $\kappa\geq 0$, and skewness parameter $\beta \in (-1,1)$.
The drift and diffusion coefficients are given by
\begin{equation}
a(x)-b(x)x 
    =
    \begin{cases}
        a_+ - b_+ x, & \text{if } x \ge \theta,\\[4pt]
        a_- - b_- x, & \text{if } x < \theta,
    \end{cases}
\qquad
\sigma(x)|x|^{\gamma(x)}
    =
    \begin{cases}
        \sigma_+ |x|^{\gamma_+}, & \text{if } x \ge \theta,\\[4pt]
        \sigma_- |x|^{\gamma_-}, & \text{if } x < \theta,
    \end{cases}
\label{eq:drift_TCKLS}
\end{equation}
where $a_-$, $\sigma_-$, and $\sigma_+$ are strictly positive constants, and $\gamma_- \in [1/2,1]$, $\gamma_+ \in [0,1]$. We have used the same notation as in \cite{mazzonetto2024parameters}.

The following theorem, Theorem~\ref{thm:strong_uniqueness_CKLS}, provides the first existence and uniqueness results for 
\begin{itemize}
    \item the T-CKLS process~\cite{mazzonetto2024parameters}, 
    \item and for its special case, the T-CIR process~\cite{decamps2006self,dong_wong_2015,zhang2024bond},
    \item the skew-sticky CIR process~\cite{zhang2024hitting}.
\end{itemize}
Moreover, our theorem provides an alternative proof of the strong existence and uniqueness for the skew CIR process~\cite{tian2018skew}, which had already been proven in~\cite{trutnau2011pathwise}, and our result substantially enlarges the class of processes among one-dimensional diffusions with sticky-skew interfaces.
In addition, our framework allows for a precise analysis of the
boundary behavior of the above process near $0$, see Remark~\ref{rem:comp_0}. 

\begin{theorem}
\label{thm:strong_uniqueness_CKLS}
There exists a unique (in law) weak solution $X$ to the SDE~\eqref{eq:sticky_skew_CKLS} for $\kappa > 0$ (skew-sticky T-CKLS), and a unique strong solution for $\kappa = 0$ (skew T-CKLS).
\end{theorem}

\begin{proof}
Let us first consider the case $\kappa>0$.

We decompose the dynamics into two auxiliary equations corresponding to the regions separated by the threshold. In this case we choose $\vartheta=\theta-2\delta$ for a given $\delta \in (0, \theta/4)$. On the lower region, we consider the classical CKLS equation
\begin{equation}
X_t = x_0 + \int_0^t (a_- - b_- X_s)\, \mathrm{d}s 
     + \sigma_- \int_0^t X_s^{\gamma_-}\, \mathrm{d}B_s,
\quad t \ge 0,
\label{eq:pasting1_SS_CIR}
\end{equation}
($\mathfrak{b}_-(x):=a_--b_- x$ and $\mathfrak{G}_-(x)=\sigma_-|x|^{\gamma_-}$) while on the upper region we introduce the auxiliary equation
\begin{equation}
\left\{
\begin{array}{ll}
X_t = x_0 + \int_0^t \mathfrak{b}_+(X_s)\,\mathds{1}_{\{X_s \neq \theta \}} \vd s
    + \int_0^t \mathfrak{G}_+(X_s)\mathds{1}_{\{X_s \neq \theta \}} \vd B_s
    + \beta L_t^{\theta}(X),  \\[6pt]
\displaystyle
\int_0^t \mathds{1}_{\{X_s \neq \theta\}}\,\vd s = \kappa L_t^{\theta}(X),
\end{array}
\right.
\label{eq:pasting2_SS_CIR}
\end{equation}
where
\begin{equation}
\mathfrak{b}_+(x) =
\begin{cases}
a_+ - b_+ x, & x \ge \theta-\delta,\\
\mathfrak{b}_-(x), & \theta-3\delta \leq x < \theta-\delta,\\
a_- - b_- (\theta-3\delta), 
& x < \theta-3\delta, 
\end{cases}
\end{equation}
\begin{equation}
\mathfrak{G}_+(x) =
\begin{cases}
\sigma_+ x^{\gamma_+}, & x \ge \theta-\delta,\\
\mathfrak{G}_-(x), & \theta-3\delta \leq x < \theta-\delta,\\
\sigma_- \left(\theta-3\delta\right)^{\gamma_-}, & x < \theta-3\delta.
\end{cases}
\end{equation}
Strong existence and uniqueness of \eqref{eq:pasting1_SS_CIR} follows from~\cite[Theorem~1]{yamada1971uniqueness}, while weak existence and uniqueness in law for~\eqref{eq:pasting2_SS_CIR} follow from~\cite[Theorem~3.4]{salins2017markov}. 
Moreover, above the interface point~$\theta$, the solution behaves as a classical CKLS process and therefore does not explode (see Proposition~\ref{prop:no_explosion_linear} in Appendix \ref{sec:app:classics}, where classical results are recalled).

Hypotheses~\hypref{hyp:pasting} and~\hypref{hyp:weak_pasting} are satisfied (with compatibility strip $(\vartheta-\delta,\vartheta+\delta)$).  
As a consequence of Proposition~\ref{prop:weak_pasting}, weak existence and uniqueness in law for~\eqref{eq:sticky_skew_CKLS} are obtained.

Let us now consider $\kappa=0$. In this case the solution process of \eqref{eq:pasting2_SS_CIR} is no longer sticky. Additionally, we can remove the indicator $\mathds{1}_{\{X_s \neq \theta\}}$ in the drift and volatility terms, as a consequence of \cite[Corollary 1.6 and Theorem 1.7, Chapter VI]{revuz2013continuous}. In fact, one can prove that any weak solution to \eqref{eq:pasting2_SS_CIR} spends no time at $x$ for all $x \in \mathbb{R}$. This means that the process is a skew process, and strong existence and uniqueness follow from \cite[Theorem 2.3]{le2006one}. 

Consequently, Hypothesis~\hypref{hyp:pasting} and Hypothesis~\hypref{hypothesis_strong} are satisfied. Theorem~\ref{thm:pasting} yields the existence of a pathwise unique global strong solution to~\eqref{eq:TCKLS}.
\end{proof}

\begin{remark}[Behavior at $0$] \label{rem:comp_0}
Our proof shows that the behavior of the process around $0$ is the same as the one of a standard CKLS process. This means that the sticky-skew T-CKLS is $\mathbb{P}$-a.s.~non-negative for all $t\geq 0$, and it reaches $0$ if and only if $\gamma_-=1/2$ and $\sigma_-^2>2a_-$, in which case $0$ is an instantaneously reflecting boundary.
\end{remark}

\appendix

\section{Background and preliminary results on SDEs} \label{sec:app:classics}

\subsection{Definitions and non-explosion}
We recall some definitions in the setting of the generalized one-dimensional SDE~\eqref{eqn:SDE_general}.

\begin{remark}[Well-definedness of the local-time term]
If $\nu$ is locally finite, then for each $t\ge0$ the integral $\int_{\mathbb R} L_t^x(X)\,\nu(\mathrm{d}x)$ is well-defined,
since $x\mapsto L_t^x(X)$ has compact support.
\end{remark}

For the definitions of weak solution, strong solution, uniqueness in law and pathwise uniqueness we refer the reader to the book of~\cite{ks}. Here, we briefly recall some notions.  

\begin{definition}[Weak solution to \eqref{eqn:SDE_general}]
A \emph{weak solution} to \eqref{eqn:SDE_general} is the triplet consisting of a filtered probability space
$(\Omega,\mathcal F,(\mathcal F_t)_{t\ge0},\mathbb P)$ supporting a 
$(\mathcal F_t)$-Brownian motion $B$ and a continuous $(\mathcal F_t)$-adapted process $X$
such that \eqref{eqn:SDE_general} holds $\mathbb P$-a.s. for all $t\ge 0$.
\end{definition}

When dealing with weak solution, we will omit the filtered probability space and only denote the pair $(X,B)$.

\begin{definition}[Strong solution to \eqref{eqn:SDE_general}]
A weak solution $(X,B)$ is called a \emph{strong solution} if $X$ is adapted to the usual
augmentation of the filtration generated by $B$.
\end{definition}

\begin{definition}[Uniqueness in law]
We say that \eqref{eqn:SDE_general} enjoys \emph{uniqueness in law} if any two weak solutions
with the same initial distribution $\mu$ have the same law $\mathbb{P}_{\mu}$ on the path space
$\mathcal C([0,\infty),\mathbb R)$.
\end{definition} 

\begin{definition}[Pathwise uniqueness]
We say that \eqref{eqn:SDE_general} enjoys \emph{pathwise uniqueness} if whenever
$(X,B)$ and $(\widetilde X,B)$ are two weak solutions defined on the same filtered probability space,
driven by the same Brownian motion $B$, and such that $X_0=\widetilde X_0$ $\mathbb{P}$-a.s.,
then $X$ and $\widetilde X$ are indistinguishable.
\end{definition}

When working with possibly explosive solutions, we use the usual localization via the exit times
$\tau_n:=\inf\{t\ge 0:\ |X_t|\ge n\}$ and the explosion time $\tau_\infty:=\lim_{n\to\infty}\tau_n$
(with the convention $\inf\emptyset=\infty$).

When $\nu\equiv 0$, equation \eqref{eqn:SDE_general} reduces to the standard SDE
\begin{equation}\label{eq:SDE_0}
    X_t = X_0 + \int_0^t b(X_s)\,\mathrm{d}s + \int_0^t \sigma(X_s)\,\mathrm{d}B_s.
\end{equation}
In this case several useful results are available for non-explosion and existence and uniqueness.

\begin{prop}[Non-explosion under linear growth {\cite[Proposition~1.5.5.3]{jeanblanc2009mathematical}}]
\label{prop:no_explosion_linear}
Assume that $b$ and $\sigma$ satisfy a linear growth condition and $X_0 \in L^2$:
there exists $K>0$ such that
\[
|b(x)|^2 + |\sigma(x)|^2 \le K(1+|x|^2), \qquad \forall x\in\mathbb{R}.
\]
Then any (local) weak solution to \eqref{eq:SDE_0} is non-explosive and $\tau_\infty=\infty$ $\mathbb P$-a.s.
\end{prop}

\begin{theorem}[Yamada-Watanabe {\cite{yamada1971uniqueness}}]
\label{thm:YW}
For the classical SDE \eqref{eq:SDE_0}:
\begin{enumerate}[label=\upshape(\roman*)]
    \item pathwise uniqueness implies uniqueness in law,
    \item weak existence together with pathwise uniqueness implies strong existence.
\end{enumerate}
\end{theorem}

These results can also be considered in the case the measure $\nu\not\equiv 0$ satisfies Hypothesis~\hypref{hyp:nu}. Indeed, in the next subsection, we show that, under quite mild assumptions on $b$ (and $\nu$) it is possible to reduce~\eqref{eqn:SDE_general} to \eqref{eq:SDE_0} and this to a driftless SDE.

\subsection{Recall of weak existence and uniqueness results} 
\label{sec:preliminary}

For our main results, we assume that weak existence and uniqueness hold for some SDE. 
In this section, we recall some results which ensure weak existence and uniqueness for the generalized SDE~\eqref{eqn:SDE_general} in the case the diffusion coefficient is not necessarily uniformly elliptic. 

First, the already mentioned~\cite[Theorem 4.48]{engel_schm_89_3} provides existence and uniqueness up to the first time the process reaches the set $E_{\sigma}$, whose definition is recalled in some lines.
There are several results based on well-posedness of martingale problems~\cite{stroock2007multidimensional}, or Dirichlet forms~\cite{fukushima2011dirichlet}.

We now recall another standard approach which relates solution to SDEs~\eqref{eqn:SDE_general} to driftless SDEs via a space transform.

Therefore, let us first consider the driftless SDE
\begin{equation}
  X_t=X_0+ \int_0^t\sigma(X_t)\vd B_t,\quad t\geq 0,
\label{SDE_drift_removed}
\end{equation}
with $\sigma$ measurable function. 
Engelbert and Schmidt proposed famous criteria for existence and uniqueness based on the following quantities:  
\begin{align}
    N_\sigma & :=\{x\in \mathbb{R} \colon \sigma(x)=0\}, \\
    E_\sigma &:=\{x\in \mathbb{R} \colon \int_{\mathcal{U}(x)}{(\sigma(y))^{-2}} \vd y=+\infty \ \forall \mathcal{U}(x) \text{ open set containing }x\}.
\end{align}

Before specifying the criteria, let us recall the important notion of {\it fundamental solution} (cf.~\cite[Definition 4.16]{engel_schm_89_3}). 
A (weak) solution is called a fundamental solution if $\sigma^2(X)>0$ on $[0, \tau_\sigma]$ for $\lambda \times \mathbb P_x$ almost every $(t,\omega)$, where $\lambda$ is the Lebesgue measure and $\tau_\sigma :=\inf\{s>0 \colon X_s \in E_\sigma \}$. 
Basically, this means that the diffusion coefficient does not vanish while the process has not reached $E_\sigma$: if $N_\sigma \subseteq E_\sigma$ this does not say much, if instead $E_\sigma \subset N_\sigma$ this provides information about the time spent on $N_\sigma\setminus E_\sigma$ until the process reaches $E_\sigma$. 
Note that if $\mathbb{P}(\tau_\sigma=\infty)=1$, then $\mathbb{P}$-a.s. the process spends no time in $N_\sigma$. 

\cite[Proposition 4.11]{engel_schm_89_3} ensures that any (possibly existing) solution to~\eqref{SDE_drift_removed} is non-explosive. (To be precise, the definition of weak solution in~\cite{engel_schm_89_3} is given up to explosion time, which is infinity in this case.) 
And~\cite[Theorem 4.17]{engel_schm_89_3} states that there exists a solution to~\eqref{SDE_drift_removed} if and only if there exists a fundamental solution,
if and only if 
$E_\sigma \subseteq N_\sigma$.
Moreover, the fundamental solution is unique in law \cite[Theorem 4.22]{engel_schm_89_3}. 
\cite[Corollary 4.23]{engel_schm_89_3} states that 
the fundamental solution is a strong Markov continuous local martingale with speed measure $m(\!\vd x) := 1/\sigma^2(x) \vd x$. The same result states that there exists a \emph{unique} solution to~\eqref{SDE_drift_removed} if and only if $E_\sigma=N_\sigma$ and the solution is a fundamental one. 
Furthermore, $E_\sigma=\emptyset$ is equivalent to the absence of absorbing points (see~\cite[Proposition 4.20]{engel_schm_89_3}).
To summarize, 
\begin{enumerate}
    [label=\textbf{\upshape ES\arabic*}]
    \item \label{item:es:0} If $N_\sigma=E_\sigma=\emptyset$, weak existence and uniqueness in law hold for~\eqref{SDE_drift_removed}. The unique solution is a fundamental solution, it is non-explosive, and has no absorbing points. 
    \item \label{item:es:equal} If $E_\sigma = N_\sigma \neq \emptyset$, weak existence and uniqueness in law 
    hold for~\eqref{SDE_drift_removed}. 
    The unique (fundamental) solution $X$ is non-explosive. Moreover, $\mathbb{P}$-a.s.~for all $t\geq 0$,
    $\int_0^{t\wedge \tau_\sigma} \ind{N_\sigma}(X_s) \vd s= \int_0^{t\wedge \tau_\sigma} \ind{E_\sigma}(X_s) \vd s=0$. 
    \item \label{item:es:diff} If $E_\sigma=\emptyset$ and $N_\sigma \neq \emptyset$, weak existence and uniqueness in law of fundamental solution hold for~\eqref{SDE_drift_removed}. The unique fundamental solution is non-explosive and has no absorbing points. Therefore the fundamental solution spends $\mathbb{P}$-a.s.~no time at the zeros of $N_\sigma$: $\mathbb{P}$-a.s.~for all $t\geq 0$ it holds that $\int_0^t \ind{N_\sigma}(X_s) \vd s =0$. 
    Note that, if for instance $\sigma=\ind{\mathbb{R}\setminus\{0\}}$, then any sticky Brownian motion with sticky interface at $0$ is a solution of the SDE~\eqref{eq:SDE_0} but not a fundamental one.
    \item \label{item:es:diff0} If $E_\sigma\neq\emptyset$ and $E_\sigma \subset N_\sigma$, weak existence and uniqueness in law of fundamental solution hold for~\eqref{SDE_drift_removed}. The unique fundamental solution $X$ is non-explosive and has absorbing points $E_\sigma$. Moreover, $\mathbb{P}$-a.s.~for all $t\geq 0$, $\int_0^{t\wedge \tau_\sigma} \ind{N_\sigma}(X_s) \vd s=0$.
\end{enumerate}
In all the above cases, the fundamental solution is a strong Markov continuous local martingale with speed measure $m(\!\vd x)$. Note that, in the literature, the speed measure may differ by a factor 2 due to different conventions.

It is well known that if there exists $(X,B)$ a solution to~\eqref{eq:SDE_0} taking values in an interval $I$ and that the measurable and locally bounded coefficients $b$ and $\sigma$ are such that $\sigma$ does not identically vanish on $I$ and
\begin{equation}\label{eq:scale_integrability}
    \frac{b}{\sigma^2}\in L^1_{\mathrm{loc}}(I),
    \qquad
    \frac{1}{\sigma^2}\in L^1_{\mathrm{loc}}(I\setminus E_\sigma),
\end{equation}
then the function 
\begin{equation}\label{eq:scale_def}
    s(x):=\exp\left(-\int_{r}^{x} 2\frac{b(u)}{\sigma^2(u)}\,\mathrm{d}u\right),
    \qquad
    S(x):=\int_{r}^{x} s(y)\vd y,
\end{equation}
for some reference point $r\in I$ are locally well defined. 
Moreover, the function $S$ is strictly increasing on each connected component where it is finite
and therefore admits a continuous inverse $R$ on its range. $S$ is called~\emph{scale function}.
Applying the generalized It\^o formula to the process $Y_t:=S(X_t)$ eliminates the drift term and yields the driftless equation
\begin{equation}\label{eq:driftless_after_scale}
    Y_t = S(x) + \int_0^t \widetilde \sigma(Y_s)\vd B_s,
    \qquad
    \widetilde \sigma(y):=(s\circ R)(y)(\sigma\circ R)(y),
\end{equation}
up to the exit time from the region where $S$ is defined.
Thanks to Engelbert and Schmidt criteria, we can study existence and uniqueness for SDE~\eqref{eq:driftless_after_scale} and then transform it back to a solution for~\eqref{eq:SDE_0}: $X=R(Y)$. Clearly there are subtle steps, and some a priori knowledge on solutions to~\eqref{eq:SDE_0} may help or be necessary to deal with the solution to~\eqref{eq:driftless_after_scale}. We illustrate these subtleties with the example of T-CKLS model in Appendix~\ref{sec:app:CKLS}.

The same approach can be implemented to remove the distributional drift coming from $\nu$ satisfying Hypothesis~\hypref{hyp:nu}. Indeed the scale function can be defined in this case as well.
In particular, we need this approach in our proof of Theorem~\ref{Theorem_uniqueness_thresh} in Appendix~\ref{ssec:app:proof:th:engeldrift}, therefore we detail it there.

We just recalled some classical results which help establishing weak existence and uniqueness (in law) for some SDEs of the kind we consider. In some specific case, we have just recalled that the solution is a Strong Markov process. In the next remark we mention that this property is common to a larger class of processes.  
\begin{remark}[Strong Markov property] \label{rem:smarkov_general}
When weak existence and uniqueness (in law) hold for the systems of SDEs considered in this paper (also the sticky ones), some classical results ensure that the unique solution 
is strong Markov with respect to the usual (augmentation of its) natural filtration. The interested reader can refer, for instance, to~\cite[Corollary 4.38]{engel_schm_89_3} or the discussion around Theorem 6.2.2 in~\cite{stroock2007multidimensional} when $\nu\equiv 0$. 
Moreover, by the theory of general linear diffusions, there exists a characterization via (possibly killed) functions of time-change of Brownian motion. This is related to the characterization of linear diffusions via \emph{speed measure}, \emph{scale function} and \emph{killing measure} which yields that the solution can be written as a function of a time-changed Brownian motion. 
\end{remark}

\section{Proofs of the Main Results} \label{sec:app:proofs}

As already mentioned, there are other results which could be used for weak existence and uniqueness via pasting approach, for instance suitable martingale problems. In this section, we provide a proof to Proposition~\ref{prop:weak_pasting} which focuses on how our Hypothesis~\hypref{hyp:pasting} intervenes. We do not insist much on technical points. The proof of Theorem~\ref{thm:pasting} is more detailed because it is our main contribution for the pasting approach. 

\subsection{Proof of Proposition~\ref{prop:weak_pasting}}

Let us fix $x\in \text{Interior}(I)$. By Hypothesis~\hypref{hyp:pasting}, $\vartheta-\varepsilon/2, \vartheta\in I_-\cap I_+$. 
It is convenient to identify the solutions to the SDEs~\eqref{eqn:SDE_1}-\eqref{eqn:SDE_2} with their law on the canonical space $\mathcal{C}([0,\infty),\RR)$, which will be denoted by $\PP_{y}^{\pm}$ when $y$ denotes the starting point.
Without loss of generality, we assume that $x\in I_-$. 

\paragraph{Proof of the weak existence.} 
Let us work on a product probability space carrying independent copies of the (unique) weak solutions to~\eqref{eqn:SDE_1}-\eqref{eqn:SDE_2}: $(X^{(-,0)},B^{(-,0)}) \sim \PP_x^-$, and 
for $n>1$,
\begin{equation}
(X^{(+,n)},B^{(+,n)}) \sim \PP_{\vartheta}^+, 
\quad 
(X^{(-,n)},B^{(-,n)}) \sim \PP_{\vartheta-\varepsilon/2}^-,
\end{equation}
all independent. 
For the formal construction of the underlying probability space and filtration,
and for general results on pasting weak solutions at stopping times,
we refer to the Appendix of \cite{lamberton2013optimal}.

Then, we define the alternating hitting times (with $\inf\emptyset=+\infty$): $\zeta_0:=\inf\{t\ge0:\ X^{(-,0)}_t=\vartheta\}$
and for $n\geq 1$,
\begin{equation}
\eta_{n}:=\inf\{t\ge0:\ X^{(+,n)}_t=\vartheta-\varepsilon/2\}, \qquad
\zeta_{n}:=\inf\{t\ge0:\ X^{(-,n)}_t=\vartheta\}.
\end{equation}
We then define
\begin{equation}
\tau_0:=0,\quad \tau_1:=\zeta_0,\quad \tau_2:=\tau_1+\eta_1,\quad \tau_3:=\tau_2+\zeta_1,\quad \ldots
\end{equation}
If one of the above $\tau_n$ is $+\infty$, i.e.~$\eta_{\lfloor n/2 \rfloor}=+\infty$ or $\zeta_{\lfloor n/2 \rfloor}=+\infty$, we stop switching from that point on.
Otherwise, on the event where all are finite, each $\tau_n$ is strictly positive by continuity of the trajectory of the process. 

Moreover, the sequences $(\eta_n)_{n\ge1}$ and $(\zeta_n)_{n\ge1}$ are i.i.d.\ within each family (and independent across families).   
Since $\eta_n$, $\zeta_n >0$ $\mathbb{P}$-a.s.\ and there exists $\varepsilon>0$ such that $\mathbb{P}(\eta_1>\varepsilon)>0$ and $\mathbb{P}(\zeta_1>\varepsilon)>0$, the Borel-Cantelli lemma implies that
$\sum_{n\ge1} \eta_n + \sum_{n\ge1} \zeta_n = +\infty \quad \mathbb{P}\text{-a.s.}$
and therefore the sequence $(\tau_n)_{n\ge0}$ cannot accumulate in finite time.

Finally, we define the global process $X$ by pasting the pieces on successive intervals:
\begin{equation}
X_t :=
\begin{cases}
X^{(-,0)}_t, & t\in[0,\tau_1),\\
X^{(+,1)}_{t-\tau_1}, & t\in[\tau_1,\tau_2),\\
X^{(-,1)}_{t-\tau_2}, & t\in[\tau_2,\tau_3),\\
\hspace{1.3em}\vdots &
\end{cases}
= \sum_{k=0}^{\infty} X^{(-,k)}_{t-\tau_{2k}} \ind{[\tau_{2k},\tau_{2k+1})}(t) 
+ \sum_{k=1}^{\infty} X^{(+,k)}_{t-\tau_{2k-1}} \ind{[\tau_{2k-1},\tau_{2k})}(t)
\end{equation}
and define the driving noise $B$ by concatenating the independent Brownian motions with the usual shift so that $B$
is continuous and has independent increments:
\begin{equation}
B_t :=
\begin{cases}
B^{(-,0)}_t, & t\in[0,\tau_1),\\
B^{(-,0)}_{\tau_1}+B^{(+,1)}_{t-\tau_1}, & t\in[\tau_1,\tau_2),\\
B^{(-,0)}_{\tau_1}+B^{(+,1)}_{\tau_2-\tau_1}+B^{(-,1)}_{t-\tau_2}, & t\in[\tau_2,\tau_3),\\
\hspace{1.3em}\vdots &
\end{cases}
= \sum_{j=0}^\infty \left( B^{(-,j)}_{t \wedge \tau_{2j+1}- t \wedge \tau_{2j} } + B^{(+,j)}_{t \wedge \tau_{2(j+1)}- t \wedge \tau_{2j+1} } \right).
\end{equation}
On each interval
the pair $(X,B)$ solves either~\eqref{eqn:SDE_2} or~\eqref{eqn:SDE_1}. Since the coefficients coincide on
$(\vartheta-\varepsilon,\vartheta+\varepsilon)$ (Hypothesis~\hypref{hyp:pasting}), this yields that $(X,B)$ solves
the threshold SDE~\eqref{eqn:SDE_general} with coefficients~\eqref{eq:coeffs}. 
Finally, each pasted segment stays in its own state space, which is either $I_-$ or $I_+$, and the switching
only occurs at $\vartheta-\varepsilon/2$ and $\vartheta$, which belong to the interior of both state spaces. By construction of $I$, we obtain
$\PP(X_t\in I\ \forall t\ge0)=1$.

\paragraph{Proof of the uniqueness in law.}
Let $(\tilde X,\tilde B)$ be any weak solution to~\eqref{eqn:SDE_general} with $\tilde X_0=x$.
Without loss of generality, assume that $x\in I_-$.
We define recursively the alternating switching times: for $n\in \mathbb{N}$, 
\begin{equation}
\tilde\tau_0:=0,\qquad
\tilde\tau_{2n+1}:=\inf\{t\ge \tilde\tau_{2n}:\ \tilde X_t=\vartheta\},\qquad
\tilde\tau_{2n+2}:=\inf\{t\ge \tilde\tau_{2n+1}:\ \tilde X_t=\vartheta-\varepsilon/2\},
\end{equation}
with the convention $\inf\emptyset= \infty$.
If $\tilde\tau_k=\infty$ for some $k$, we set $\tilde\tau_j:=\infty$ for all $j\ge k$.

By continuity of $\tilde X$, for every $n\ge0$,
\[
\tilde X_t<\vartheta\quad \text{for }t\in[\tilde\tau_{2n},\tilde\tau_{2n+1}),
\qquad
\tilde X_t>\vartheta-\varepsilon/2\quad \text{for }t\in[\tilde\tau_{2n+1},\tilde\tau_{2n+2}).
\]
Hence, on $[\tilde\tau_{2n},\tilde\tau_{2n+1})$ the coefficients of~\eqref{eqn:SDE_general} coincide along the path with $(\mathfrak b_-,\mathfrak G_-)$, while on $[\tilde\tau_{2n+1},\tilde\tau_{2n+2})$ they coincide along the path with $(\mathfrak b_+,\mathfrak G_+)$. The overlap assumption ensures that there is no ambiguity at the interface.

We claim that, for every $k\ge0$, the law of the stopped process
$
    (\tilde X_{t\wedge \tilde\tau_k})_{t\ge0}
$ 
is uniquely determined. For $k=1$, this follows from weak uniqueness for the auxiliary equation on $I_-$, since up to $\tilde\tau_1$ the process evolves in the region where the coefficients agree with $(\mathfrak b_-,\mathfrak G_-)$.

Assume now that the law of $(\tilde X_{t\wedge \tilde\tau_k})_{t\ge0}$ is uniquely determined for some $k\ge1$. Consider the conditional law of the shifted pair
$
    \bigl(\tilde X_{\tilde\tau_k+t},\,\tilde B_{\tilde\tau_k+t}-\tilde B_{\tilde\tau_k}\bigr)_{t\ge0}
$
given $\mathcal F_{\tilde\tau_k}$. Up to the next switching time $\tilde\tau_{k+1}$, this conditional law is a weak solution of the relevant auxiliary SDE started from $\tilde X_{\tilde\tau_k}$: if $k$ is odd, then $\tilde X_{\tilde\tau_k}=\vartheta$ on $\{\tilde\tau_k<\infty\}$ and the coefficients coincide with $(\mathfrak b_+,\mathfrak G_+)$ until $\tilde\tau_{k+1}$; if $k$ is even, then $\tilde X_{\tilde\tau_k}=\vartheta-\varepsilon/2$ on $\{\tilde\tau_k<\infty\}$ and the coefficients coincide with $(\mathfrak b_-,\mathfrak G_-)$ until $\tilde\tau_{k+1}$. By weak uniqueness for the corresponding auxiliary equation, the conditional law of the segment between $\tilde\tau_k$ and $\tilde\tau_{k+1}$ is therefore uniquely determined. Since the law up to $\tilde\tau_k$ is unique by the induction hypothesis, the law up to $\tilde\tau_{k+1}$ is unique as well.

Thus, by induction, for every $k\ge0$ the law of $(\tilde X_{t\wedge \tilde\tau_k})_{t\ge0}$ is uniquely determined. Moreover, the sequence $(\tilde\tau_k)_{k\ge0}$ cannot accumulate in finite time: indeed, if $\tilde\tau_k\to T<\infty$, then by continuity
\[
\tilde X_{\tilde\tau_{2n+1}}=\vartheta \to \tilde X_T,
\qquad
\tilde X_{\tilde\tau_{2n+2}}=\vartheta-\varepsilon/2 \to \tilde X_T,
\]
which is impossible since $\vartheta\neq \vartheta-\varepsilon/2$. Hence $\tilde\tau_k\uparrow\infty$ $\mathbb{P}$-a.s. Therefore, by passage to the limit, the full law of $\tilde X$ is uniquely determined. In particular, it coincides with the law of the concatenated process constructed in the proof of weak existence.

\subsection{Proof of Theorem \ref{thm:pasting}}
Given a probability space, endowed with a Brownian motion $B$, we denote, respectively, by $X^{(+,y,B)}$ and $X^{(-,y,B)}$ the strong solution to the SDEs~\eqref{eqn:SDE_1} and \eqref{eqn:SDE_2} starting from a given $y$ in the interior of the respective state space and driven by the given Brownian motion $B$. 
Since Hypothesis~\hypref{hyp:pasting} holds, the threshold $\vartheta$ lies in the state space of both processes, this is the case for $\vartheta-\varepsilon/2$ as well.  

From now on, we consider a probability space endowed with a Brownian motion $B$. 

Assume, without loss of generality, that $x\in \text{Interior}(I) \cap (-\infty,\vartheta)$. If $x\in \text{Interior}(I) \cap [\vartheta,+\infty)$, the proof is analogous. 

\paragraph{Proof of the strong existence.} 

\textit{Step 1: Definition of the alternating stopping times.}
Let us define the sequence of alternating stopping times $(\tau_n)_{n \ge 0}$ recursively as follows. (We use the convention $\inf\emptyset=+\infty$.)

\begin{enumerate}[label=(\roman*)]
\item Set $\tau_0:=0$, $B^{(1)}:=B$ and define
\begin{equation}
\tau_1 := \inf\{t \ge 0 : X^{(-,x,B^{(1)})}_{t-\tau_0} \notin (-\infty, \vartheta]\},
\end{equation}
the first exit time of $(-\infty,\vartheta]$ for $X^{(-,x,B^{(1)})}$ (by continuity, it coincides with the hitting time (from below) of $\vartheta$). 

\item Next, if $\tau_1=+\infty$, we stop, otherwise set $B^{(2)}_t := B_{t+\tau_1} - B_{\tau_1}$ and 
\begin{equation}
\tau_2 := \inf\{t \ge \tau_1 : X^{(+, \vartheta,B^{(2)})}_{t-\tau_1} \notin [\vartheta-\varepsilon/2, \infty)\},
\end{equation}
the first hitting time of $X^{(+, \vartheta,B^{(2)})}$ at the point $\vartheta-\varepsilon/2$ after time $\eta_1$.

\item Proceeding inductively, for $n \ge 1$, define
\begin{equation}
\tau_{2n+1} := \inf\{t \ge \tau_{2n} : X^{(-, \vartheta- \varepsilon/2,B^{(2n+1)})}_{t-\tau_{2n}} \notin (-\infty, \vartheta]\},
\end{equation}
if $\tau_{2n}<\infty$, and if $\tau_{2n+1}<\infty$ then
\begin{equation}
\tau_{2n+2} := \inf\{t \ge \tau_{2n+1} : X^{(+, \vartheta,B^{(2n+2)})}_{t-\tau_{2n+1}} \notin [\vartheta-\varepsilon/2, \infty)\},
\end{equation}
where for each $j \in \mathbb{N}$, 
$ B^{(j+1)}_t := B_{t+\tau_j} - B_{\tau_j}$.
\end{enumerate}

The sequence $(\tau_n)_{n\ge0}$ is strictly increasing and, as shown in the proof of Proposition~\ref{prop:weak_pasting}, cannot accumulate in finite time. 
Hence, for any finite horizon $T\in[0,+\infty)$, only finitely many $\tau_n$ fall below $T$ almost surely:
\begin{equation}
    N(T):=\max\{j\ge0:\tau_j\le T\}<+\infty \quad \text{$\mathbb{P}$-a.s.}
\end{equation}

\textit{Step 2: Construction of the local solution pieces.}
We now construct the sequence of processes $(Y^{(n)})_{n \ge 1}$ on the corresponding time intervals:
\begin{equation}
Y^{(1)}_t := X^{(-,x,B^{(1)})}_{t-\tau_0}, 
\quad t \in [\tau_0, \tau_1],
\qquad \text{and} \qquad 
Y^{(2)}_t := X^{(+, \vartheta,B^{(2)})}_{t-\tau_1}, 
\quad t \in [\tau_1, \tau_2],
\end{equation}

and, in general, for \(n \ge 1\),

\begin{equation}
Y^{(2n+1)}_t := X^{(-, \vartheta-\varepsilon/2,B^{(2n+1)})}_{t-\tau_{2n}}, 
\quad t \in [\tau_{2n}, \tau_{2n+1}],
\end{equation}

\begin{equation}
Y^{(2n+2)}_t := X^{(+, \vartheta,B^{(2n+2)})}_{t-\tau_{2n+1}}, 
\quad t \in [\tau_{2n+1}, \tau_{2n+2}].
\end{equation}

An analogous construction applies when \(x \in [\vartheta, \infty) \cap \text{Interior}(I)\): in that case, \(\tau_1\) is defined as the first exit time of 
$X^{(+,x,B^{(1)})}$ from \((\vartheta-\varepsilon/2, \infty)\), 
and so on.

\textit{Step 3: Pasting the pieces together.}
We finally define the process \(X\) by concatenating the pieces \(\big(Y^{(j)}\big)_{j\ge1}\):
\begin{equation}
    X_t := \sum_{j = 0} Y^{(j)}_t\,\mathbf{1}_{[\tau_{j},\tau_{j+1})}(t),
\qquad t \ge 0.
\label{eq:sol:concat}
\end{equation}

By construction, and since each \(Y^{(j)}\) has continuous sample paths, the pieces match at the switching times:
for every \(n\ge1\),
\[
Y^{(2n)}_{\tau_{2n}} = Y^{(2n+1)}_{\tau_{2n}}= \vartheta-\varepsilon/2,
\qquad
Y^{(2n-1)}_{\tau_{2n-1}} = Y^{(2n)}_{\tau_{2n-1}} = \vartheta.
\]
Hence \(X\) has continuous sample paths.

Fix \(t\ge0\). As we mentioned above \(N(t)=\max\{j\ge0:\tau_j\le t\}\) is $\mathbb{P}$-a.s.\ finite. 
On each interval \([\tau_{j-1},\tau_j)\), the process \(X\) coincides with a strong solution to either
\eqref{eqn:SDE_1} or \eqref{eqn:SDE_2} driven by the corresponding shifted Brownian motion.
Using that \(b,\sigma\) agree with the relevant local coefficients \(\mathfrak{b}_\pm,\mathfrak{G}_\pm\) (and that Hypothesis~\hypref{hyp:pasting}
ensures consistency on the overlap), we may write the integral form of~\eqref{eqn:SDE_general} on each subinterval and sum over
\(j=0,\dots,N(t)\), obtaining the drift and stochastic integral terms up to time \(t\).

For the local-time drift, we use additivity over time intervals: for each fixed \(y\in\mathbb R\),
\[
    L_t^y(X)
    =\sum_{j=0}^{N(t)} \Big(L_{t\wedge\tau_{j+1}}^y(X)-L_{\tau_{j}}^y(X)\Big).
\]
Note that $t < \tau_{N(t)+1}$. 
Since the sum is finite and \(\nu\) is locally finite, we can integrate term-by-term:
\[
    \int_{\mathbb R}L_t^y(X)\,\nu(\mathrm{d}x)
    =
    \sum_{j=0}^{N(t)}\int_{\mathbb R}\Big(L_{t\wedge\tau_{j+1}}^y(X)-L_{t\wedge\tau_{j}}^y(X)\Big)\nu(\mathrm{d}y).
\]
On each subinterval $I_-$ and $I_+$, the path remains in the corresponding region, so only the restriction of \(\nu\) to that region contributes,
the overlap condition \eqref{eq:overlap_nu} in Hypothesis~\hypref{hyp:pasting} guarantees that the local-time term is consistent when \(X\) lies in
\((\vartheta-\varepsilon,\vartheta+\varepsilon)\).
Therefore the concatenated process \(X\) satisfies~\eqref{eqn:SDE_general} at time \(t\).

Finally, since each local strong solution is a measurable functional of the driving Brownian path on its interval,
each piece \(Y^{(j)}\) is adapted to \((\mathcal F^B_{\tau_{j-1}+s})_{s\ge0}\). Consequently,
the concatenated process \(X\) is adapted to the completed filtration generated by \(B\), and hence is a (global) strong solution
to~\eqref{eqn:SDE_general} for the given \(x\in\mathrm{Interior}(I)\).

\paragraph{Proof of the pathwise uniqueness.}
Let $X$ and $\tilde{X}$ be two solutions of the SDE~\eqref{eqn:SDE_general} with the same initial condition and Brownian motion $B$ on the same probability space. 
Without loss of generality, assume that $X_0 =\tilde X_0 < \vartheta$. 
We can construct sequences of stopping times as we did in the proof of the strong existence. Denote by $\eta_1$ (resp.\ $\tilde{\eta}_1$) the first hitting time of $X$ (resp.\ $\tilde{X}$) of the point $\vartheta$. 
Define $\eta_2$ (resp.\ $\tilde{\eta}_2$) as the first hitting time of $\vartheta-\varepsilon/2$ after time $\eta_1$ 
(resp.\ $\tilde{\eta}_1$). 
Subsequent stopping times are defined similarly by alternating between these two intervals.

Since $X_{\eta_1} = \tilde{X}_{\tilde{\eta}_1} = \vartheta$ $\mathbb{P}$-a.s., both processes satisfy the same SDE as $X^{(-,x,B)}$ up to the times $\eta_1$ and $\tilde{\eta}_1$. 
By strong existence and uniqueness for this SDE, we deduce that
\begin{equation}
\eta_1 = \tilde{\eta}_1 \quad \mathbb{P}-\text{a.s.}, 
\qquad 
X_{t\wedge\eta_1} = \tilde{X}_{t\wedge\tilde{\eta}_1},\quad t\ge0.
\end{equation}

From now on we identify the two stopping times and simply write $\eta_1$.

Define
\begin{equation}
B^{(1)}_t := B_{t+\eta_1} - B_{\eta_1}, \qquad t\ge0.
\end{equation}
By the strong Markov property of Brownian motion, $B^{(1)}$ is a Brownian motion with respect to the shifted filtration 
$(\mathcal{F}^B_{t+\eta_1})_{t\ge0}$, independent of $(B_t)_{t\le \eta_1}$.

Both shifted processes $(X_{t+\eta_1})_{t\ge0}$ and $(\tilde X_{t+\eta_1})_{t\ge0}$
solve SDE~\eqref{eqn:SDE_1} as $X^{(+,\vartheta,B^{(1)})}$ driven by $B^{(1)}$ with the same initial condition $\vartheta$.
By strong existence and pathwise uniqueness for this SDE, we deduce that
\begin{equation}
\eta_2 = \tilde{\eta}_2 \quad \mathbb{P}-\text{a.s.}, 
\qquad 
X_{t+\eta_1} = \tilde{X}_{t+\eta_1} \quad\text{on } [0,\eta_2-\eta_1].
\end{equation}

With the same argument used in the proof of the weak existence in Proposition~\ref{prop:weak_pasting} it can be shown that there is no accumulation of the stopping times. By induction, we can show that $X$ and $\tilde{X}$ are indistinguishable on $[0,T]$.

\subsection{Proof of Proposition~\ref{prop_uniqueness_thresh}}

The following proof is an extension of the result stated in~\cite[Theorem 1.3]{le2006one} or in~\cite[Theorem~IX.3.5.(iii)]{revuz2013continuous}.
Let us first recall the following useful result, whose proof is based on the occupation times formula and continuity properties of the right local time $x \mapsto \ell_t^x(X)$.

\begin{lemma}[{cf.,e.g.,\cite[Lemma~1.1]{le2006one}~\cite[Lemma IX.3.3]{revuz2013continuous}}]
\label{lemma:LT0}
Let $\epsilon>0$ be fixed and assume that $X$ is a continuous semimartingale such that for every $t>0$:
\begin{equation}
\int_0^t \frac{\mathds{1}_{(0,\epsilon)}(X_s)}{\varrho(X_s)}\vd\langle X,X\rangle_s <\infty \quad \mathbb{P}-\text{a.s.}
\label{A_t_finite}
\end{equation}
for a measurable function $\varrho \colon (0,\infty) \to (0,\infty)$ such that $\int_0^{\epsilon} 1/\varrho(x) \vd x =+\infty$.
Then the right local time $\ell_\cdot^0(X)=0$ $\mathbb{P}$-a.s.. 
\label{Lemme_local1}
\end{lemma}

Let $\epsilon \in (0, \varepsilon]$. Without loss of generality, we assume that $\rho$ satisfies $\int_0^{\epsilon} 1/\rho(x) \vd x =+\infty$. If this is not the case, then 
$\int_{-\epsilon}^0 1/\rho(x) \vd x =+\infty$
and the proof works the same, just by taking $X:= X^{(2)} -X^{(1)}$ instead of taking $X:= X^{(1)} -X^{(2)}$ in the following proof. 

We can reduce to the case $b$ and $\sigma$ are bounded.   
For $n\in \mathbb N$ an $i=1,2$, let $\tau^{(i)}_n$ denote $\tau_n^{(i)}:=\inf\{t\in [0,\infty) \colon |X^{(i)}|\geq n\}$ and let $\tau_n=\min \{\tau^{(1)}_n,\tau^{(2)}_n\}$.
Note that for all $n\in \mathbb N$ the processes $(X^{(1)}_{s\wedge \tau_n})_{s\in [0,t]}$ and $(X^{(2)}_{s\wedge \tau_n})_{s\in [0,t]}$ satisfy~\eqref{eqn:SDE_general} on the event $\{t<\tau_n\}$ and have bounded coefficients.  

Fix $m\in \mathbb N$ and let $K_m:=[-m,m]\setminus U$ which is compact. Let $X:=X^{(1)}-X^{(2)}$, we prove that, for every $t>0$, $\ell_{t\wedge \tau_m}^0(X)=0$ $\mathbb{P}$-a.s.~using Lemma \ref{lemma:LT0}.
We prove that
\begin{equation}
    \forall t>0, \quad A_t:=\int_0^{t\wedge \tau_m} \frac{\mathds{1}_{(0,\epsilon)}(X_s)}{\rho(X_s)}\vd\langle X,X\rangle_s <\infty\quad  \mathbb{P}-\text{a.s.}
\end{equation}

We can rewrite $A_t$ as the sum of two terms
\begin{equation}
 A_t^1:=\int_0^{t\wedge \tau_m}\frac{|\sigma(X^{(1)}_s)-\sigma(X^{(2)}_s)|^2}{\rho(X_s)} \mathds{1}_{(0,\epsilon)}(X_s)\mathds{1}_{\{ X^{(2)}_s, X^{(1)}_s \in U^{\varepsilon}\}} \vd s, 
\end{equation}
\begin{equation}
 A_t^2:=\sum_{i=1}^2\int_0^{t\wedge \tau_m}\frac{|\sigma(X^{(1)}_s)-\sigma(X^{(2)}_s)|^2}{\rho(X_s)} \mathds{1}_{(0,\epsilon)}(X_s) \mathds{1}_{[-m,m]\setminus U^{\varepsilon}} (X^{(i)}_s)\vd s.
\end{equation}
Note that
\begin{equation}
A_t^2 \leq A_t^3:=2\int_0^{t\wedge \tau_m} \frac{|\sigma(X^{(1)}_s)-\sigma(X^{(2)}_s)|^2}{\rho(X_s)} \mathds{1}_{(0,\epsilon)}(X_s) \mathds{1}_{K_m}(X^{(1)}_s) \mathds{1}_{K_m}(X^{(2)}_s)\vd s.
\end{equation}
It suffices to prove that for $i=1,2,3$ and for all $t>0$, $A_t^i<\infty$ $\mathbb{P}$-a.s.. 
Let us first fix an arbitrary $t>0$. By~\ref{item:hyp:1}, $A_t^1 \leq t <+\infty$ $\mathbb{P}$-a.s.. 
To show that $A_t^2<\infty$ $\mathbb{P}$-a.s., we instead show that $A_t^3<\infty$ $\mathbb{P}$-a.s.. 
From now on, we let $\delta<\epsilon$ be fixed (that later will tend to 0) and denote 
\begin{equation}
\Omega_{\delta}(s):=\{X^{(1)}_s, X^{(2)}_s \in K_m, X_s\in(\delta,\epsilon)\}.
\end{equation}
\ref{item:hyp:2} and dominated convergence ensure that 
\begin{equation}
    \mathbb{E}[A_t^{3}] \leq \lim_{\delta \to 0^+} M_{t,\delta}(f)
\end{equation}
with
\begin{equation}
    M_{t,\delta}(f):=\mathbb{E}\left[\int_0^{t\wedge \tau_m}\frac{f(X^{(1)}_s)-f(X^{(2)}_s)}{X_s}\mathds{1}_{\Omega_{\delta}(s)}\vd s\right].
\end{equation}
Without loss of generality we can consider $f$ as a function bounded and increasing from the entire $\mathbb{R}$. Let us consider a sequence $(f_n)_{n\geq 0}$ of uniformly bounded functions and such that, for all $n\in\mathbb{N}$, $f_n\in\mathcal{C}^1(\mathbb{R})$ and $f_n(x)\xrightarrow[n\to \infty]{}f(x)$ for $x$ in the set of continuity of $f$. As the set of discontinuity points of $f$ is negligible with respect to the Lebesgue measure, we easily prove that $M_{t,\delta}(f_n)\xrightarrow[n\to \infty]{}M_{t,\delta}(f)$.

Under the above notations, the fundamental theorem of calculus and Fubini ensure that 
\begin{equation}
    \begin{split}
    M_{t,\delta}(f_n) 
    = \mathbb{E}\left[\int_0^1 \int_0^{t\wedge \tau_m}  f_n'(Z^u_s)  \mathds{1}_{ \Omega_{\delta}(s)}\vd s \vd u\right]
    \end{split}
\end{equation}
where $Z^u_s=X^{(2)}_s+u(X^{(1)}_s-X^{(2)}_s)$.
For every $u\in [0,1]$ $Z^u_s$ satisfies:
\begin{equation}
    Z_s^u=Z_0^u+\int_0^s b^{(u)}_v \vd v+\int_0^s \sigma^{(u)}_v \vd B_v,
\end{equation}
with
\begin{equation}
    b^{(u)}_v := b(X^{(2)}_v) + u (b(X^{(1)}_v) - b(X^{(2)}_v)) \quad \text{and}
\end{equation}
\begin{equation}
    \sigma^{(u)}_v:= \sigma(X^{(2)}_v) + u (\sigma(X^{(1)}_v) - \sigma(X^{(2)}_v)) 
    = (1-u)\sigma(X^{(2)}_v)  + u\, \sigma(X^{(1)}_v).
\end{equation}

Assumption~\ref{item:hyp:2} 
ensures that, on the event $\Omega_{\delta}(s)$, 
$ \sigma^{(u)}_s \geq \inf_{x\in K_m} \sigma(x) =: \alpha \in (0,+\infty)$ and $ Z^u_s\in [-m,m]$. 

We now show that 
\begin{equation}
    \sup_{x\in [-m,m], u\in[0,1]}     \mathbb{E}\left[\ell_{t\wedge \tau_m}^x(Z^u)\right]
    \leq \sup_{x\in \mathbb{R}, u\in[0,1]} \mathbb{E}\left[\ell_{t\wedge \tau_m}^x(Z^u)\right] <\infty.
\end{equation}
First note that $\sup_{u\in [0,1]} \mathbb{E}\left[|Z^{u}_{t\wedge \tau_m}|\right]<+\infty$. 
Next, Tanaka formula implies that for all $x\in \mathbb{R}$, $u\in [0,1]$, it holds $\mathbb{P}$-a.s. that
\begin{equation}
    \ell^{x}_t(Z^{u}) 
    = |Z^{u}_t-x| - |Z^{u}_0-x| - \int_0^t \sgn(Z^{u}_s-x) b^{(u)}_s \vd s - \int_0^t \sgn(Z^{u}_s-x) \sigma^{(u)}_s \vd B_s,
\end{equation}
with $\operatorname{sgn}(x)=\mathds{1}_{x>0}-\mathds{1}_{x\leq 0}$.

Since $\sigma|_{[-m,m]}$ is bounded, then $\sigma^{(u)}$ is bounded, then the stochastic integral is a martingale of mean 0, hence for all $x\in \mathbb{R}$, $u\in [0,1]$ 
\begin{equation} \label{eq:bound:L}
    \begin{split} 
    \mathbb{E} \left[\ell^{x}_{t\wedge \tau_m}(Z^{u}) \right] 
    & \leq 
    \mathbb{E}\left[ |Z^{u}_{t\wedge \tau_m}-Z^{u}_0| \right] +  \int_0^{t\wedge \tau_m} \mathbb{E}\left[ | b^{(u)}_s | \right]\vd s 
    \\
    &    \leq \sup_{v\in [0,1]}   \mathbb{E}\left[ |Z^{v}_{t\wedge \tau_m}|  \right]  +m+ t \|b|_{[-m,m]}\|_\infty <+\infty.
    \end{split}
\end{equation}
Since the right-hand-side of the latter inequality does not depend on $u$ nor on $x$, we deduce the desired bound. 

Then, using the occupation times formula, we have:
\begin{equation} 
	\begin{split}
    M_{t,\delta}(f_n)
    & \leq 
    \mathbb{E}\left[\int_0^1 \int_0^{t\wedge \tau_m} f_n'(Z_s^u) \frac{(\sigma^{(u)}_s)^2}{(\sigma^{(u)}_s)^2} \mathds{1}_{\Omega_{\delta}(s)} \vd s \vd u\right]\\
    &  \leq \frac1{\alpha^2} \mathbb{E}\left[\int_0^1 \int_0^{t\wedge \tau_m} f_n'(Z_s^u) \mathds{1}_{[-m,m]}(Z^u_s) \vd\langle Z^u,Z^u\rangle_s \vd u\right]
    \\
    &\leq \frac1{\alpha^2} \int_0^1\mathbb{E}\left[\int_{-m}^m f_n'(x)\ell_{t\wedge \tau_m}^x(Z^u)\vd x\right]\vd u\leq C,
	\end{split}
\end{equation}
where $C$ is a strictly positive constant (not depending on $n$ nor on $\delta$). In the latter inequalities we used the fact that the local time is bounded and $\int_{-m}^m f_n'(x)\vd x = f_n(m)-f_n(-m) \leq 2 \|f_n\|_\infty<+\infty$. Thus, since $M_{t,\delta}(f) =\lim_n M_{t,\delta}(f_n)$, $M_{t,\delta}(f)$ is bounded by a constant (depending on $m$ but not on $\delta$). 
Therefore $\mathbb{E}[A_t^{3}]  <+\infty$. This implies that  $A_t^3<+\infty$ $\mathbb{P}$-a.s..

Thus, $A_t<\infty$ a.s, so Lemma \ref{Lemme_local1} ensures that $\mathbb{P}$-a.s.~for all $t>0$, $\ell_{t\wedge \tau_m}^0(X^{(1)}-X^{(2)})=0$. 
Since there is no explosion for $X^{(1)}$ nor $X^{(2)}$, $\tau_m$ tends $\mathbb{P}$-a.s.~to $+\infty$ when $m$ goes to infinity. 
We conclude that $\ell_{\cdot}^0(X^{(1)}-X^{(2)})=0$ $\mathbb{P}$-a.s.. 

The proof is thus completed.  

\subsection{Proof of Theorem~\ref{Theorem_uniqueness_thresh}} \label{ssec:app:proof:th:engeldrift}

\subsubsection{When $\nu\equiv 0$.}

The proof consists in a classical reasoning which implies that pathwise uniqueness holds, see, e.g.~\cite[Proposition 3.1]{revuz2013continuous}:
Combining Proposition~\ref{prop_uniqueness_thresh} and
Tanaka formula ensures that $X^{(1)}\vee X^{(2)}$ and $X^{(1)} \wedge X^{(2)}$ are two solutions to~\eqref{eqn:SDE_general} with $\nu\equiv 0$. The fact that uniqueness in law holds and continuity of paths implies that $X^{(1)}$ and $X^{(2)}$ are indistinguishable.

\subsubsection{When $\nu\not\equiv 0$.} \label{sec:proof:PU}

Let $I$ be the state space of the solution (unique in law). For simplicity, we keep denoting by $\nu,b,\sigma$ their restricted versions to $I$.

We denote by $\nu=\nu^{\mathrm{na}}+\nu^{\mathrm{at}}$, where $\nu^{\mathrm{at}}:=\sum_{a\in I}\nu(\{a\})\delta_a$
and $\nu^{\mathrm{na}}$ is atomless. Fix $x_*\in \text{Interior}(I)$ and define $g_\nu:I\to(0,\infty)$ by
\begin{equation}\label{eq:g_nu}
g_\nu(x):=
\begin{cases}
\exp\!\big(-2\,\nu^{\mathrm{na}}((x_*,x])\big)
\displaystyle\prod_{a\in(x_*,x]}\frac{1-\nu(\{a\})}{1+\nu(\{a\})},
& x\ge x_*,\\[1.1em]
\exp\!\big(2\,\nu^{\mathrm{na}}((x,x_*])\big)
\displaystyle\prod_{a\in(x,x_*]}\frac{1+\nu(\{a\})}{1-\nu(\{a\})},
& x< x_*,
\end{cases}
\end{equation}
(products over the atoms of $\nu$ in the interval). Define
\begin{equation}\label{eq:Phi_nu}
\Phi_\nu(x):=\int_{x_*}^x g_\nu(u)\,\vd u,\qquad x\in I,
\end{equation}
and set $J:=\Phi_\nu(I)$ and $\Psi_\nu:=\Phi_\nu^{-1}:J\to I$.

Moreover, for all $x\in J$, let 
\begin{equation}
    \widetilde b(x)=g_\nu \circ \Psi_\nu(x) \, b \circ \Psi_\nu(x)
\quad\text{   and   }\quad
    \widetilde \sigma(x)=g_\nu \circ \Psi_\nu(x) \, \sigma \circ \Psi_\nu(x).
\end{equation}

\begin{lemma} \label{lem:nu_coeff_new}
  If $(b,\sigma,\nu)$ satisfy assumption~\hypref{hyp:engel-schm:modif}, then $(\widetilde b, \widetilde \sigma)$, defined in Lemma~\eqref{lem:nu_removal}, satisfy the version of~\hypref{hyp:engel-schm:modif} with $\nu\equiv 0$.  
\end{lemma}
    The proof of Lemma~\ref{lem:nu_coeff_new} is a quite direct consequence of the following properties of $g_\nu, \Phi_\nu, \Psi_\nu$. It is therefore omitted. 
    Since $|\nu(\{x\})|< 1$ for all $x\in\mathbb R$, then $g_\nu$ is strictly positive and c\`adl\`ag on $I$,
    has locally bounded variation and is locally bounded. Moreover, for every compact $K\subset I$
    there exist constants $0<m_K\le M_K<\infty$ such that $m_K\le g_\nu\le M_K$ on $K$.
    Consequently $\Phi_\nu$ and $\Psi_\nu$ are bi-Lipschitz on compacts.

The following lemma is a direct consequence of It\^o-Tanaka formula and the properties of $\Phi_\nu$. The proof is therefore omitted. 
\begin{lemma}[Change of variables]\label{lem:nu_removal}
Let Hypothesis~\hypref{hyp:nu} and assume that $b$ and $\sigma$ are measurable and locally bounded.
For every $(X,B)$, weak non-explosive solution to~\eqref{eqn:SDE_general}, let $I$ denote the smallest interval such that~$X_t\in I$ for all $t\geq 0$ $\mathbb{P}$-a.s.. Then $(\Phi_\nu(X),B)$ is a weak non-explosive solution to the classical SDE
\begin{equation}\label{eq:SDE_no_nu}
Y_t=\Phi_\nu(x_0)+\int_0^t \widetilde b(Y_s)\,\vd s+\int_0^t \widetilde\sigma(Y_s)\,\vd B_s,
\qquad t\ge 0.
\end{equation}
Moreover $\Phi_\nu(X)\in J$ for all $t\geq 0$ $\mathbb{P}$-a.s..

Conversely, if $(Y,B)$ is a weak non-explosive solution to~\eqref{eq:SDE_no_nu} and $Y_t\in J$ for all $t\geq 0$ $\mathbb{P}$-a.s., then $(\Psi_\nu(Y),B)$ is a weak non-explosive solution to~\eqref{eqn:SDE_general} and $\Psi_\nu(Y_t)\in I$ for all $t\geq 0$ $\mathbb{P}$-a.s.. 
\end{lemma}

\begin{remark}[Transfer of existence and uniqueness results]\label{rem:transfer_wellposed}
The map $x\mapsto \Phi_\nu(x)$ is a deterministic strictly increasing bijection between
$\mathcal C([0,\infty),I)$ and $\mathcal C([0,\infty),J)$ and so $x\mapsto \Psi_\nu(x)$ from $\mathcal C([0,\infty),J)$ and $\mathcal C([0,\infty),I)$. 
Hence weak existence, non-explosion, uniqueness in law, pathwise uniqueness and strong existence transfer
between \eqref{eqn:SDE_general} (on $I$) and \eqref{eq:SDE_no_nu} (on $J$).
\end{remark}

\begin{proof}[Proof of Theorem~\eqref{Theorem_uniqueness_thresh}]
Let $(X^{(1)},B^{(1)})$ and $(X^{(2)},B^{(2)})$ be two weak solutions to~\eqref{eqn:SDE_general}.
Since uniqueness in law holds, the respective state spaces coincide, and we denote them by $I$. Note that $I\subseteq \mathbb{R}$ because of non-explosion.  
As above, one can construct $\Phi_\nu$, $\Psi_\nu$, 
and coefficients $\widetilde \nu\equiv 0$, $\widetilde b$, and $\widetilde \sigma$ which satisfy Hypothesis~\hypref{hyp:engel-schm:modif} by
Lemma~\ref{lem:nu_coeff_new}.
By Lemma~\ref{lem:nu_removal}, assumption~\hypref{hyp:pathwise} is also satisfied. 
Since uniqueness in law holds for~\eqref{eqn:SDE_general}, then it also holds for~ \eqref{eq:SDE_no_nu}, and Theorem~\ref{Theorem_uniqueness_thresh} for $\nu\equiv 0$ ensures that $\Phi_\nu(X^{(1)})$ and $\Phi_\nu(X^{(2)})$ are indistinguishable. Remark~\ref{rem:transfer_wellposed}, $X^{(1)}=\Psi_\nu \circ \Phi_\nu (X^{(1)})$ and 
$X^{(2)}=\Psi_\nu \circ \Phi_\nu (X^{(2)})$ are indistinguishable. 
The proof is thus completed.
\end{proof}

\subsection{Proof of Theorem~\ref{Theorem_comparison}}

\subsubsection{When $\nu\equiv 0$.}
This proof follows a similar local-time argument as in Proposition~\ref{prop_uniqueness_thresh},
as in \cite[Theorem~1.4]{le2006one}.
    
Assume that the coefficients $b_1$, $b_2$, $\sigma$ are bounded and also assume that $b_1$ is Lipschitz with Lipschitz constant $K_b$.
	
Let us first note that following the proof of Proposition~\ref{prop_uniqueness_thresh} we can prove that $\ell^0_t(X^{(1)}_t-X^{(2)}_t) =0$. The only difference consists in the drift coefficient of $Z^u$ which is 
\begin{equation}
    		b^{(u)}_\cdot := b_2(X^{(2)}_\cdot) + u (b_1(X^{(1)}_\cdot) - b_2(X^{(2)}_\cdot)).
\end{equation}
But this does not change any argument of the proof. 
	
By It\^o-Tanaka formula, and the fact that $\ell^0_t(X^{(2)}_t-X^{(1)}_t) =0$, we have that 
\begin{equation}
        X_t:=(X^{(2)}_t-X^{(1)}_t)^+:=\max\{0, X^{(2)}_t-X^{(1)}_t\}
\end{equation}
satisfy $\mathbb{P}$-a.s.~for all $t\geq 0$: 
	\begin{equation}
		X_t
		:= X_0
		+ \int_0^t \ind{X_s \geq 0} ( b_2(X^{(2)}_s) - b_1(X^{(1)}_s) ) \vd s 
		+ \int_0^t \ind{X_s \geq 0} ( \sigma(X^{(2)}_s) - \sigma(X^{(1)}_s) ) \vd B_s. 
	\end{equation}
	Note that for all $s\geq 0$   
	\begin{equation}
		\ind{X_s \geq 0} ( b_2(X^{(2)}_s) - b_1(X^{(1)}_s) ) 
		\leq  \ind{X_s \geq 0} ( b_1(X^{(2)}_s) - b_1(X^{(1)}_s) ) 
		\leq K_b X_s.
	\end{equation}
	Since $\sigma$ is bounded, the stochastic integral has null expectation. 
	Observe that $\mathbb{P}(X_0 =0)=1$. 
	Therefore, for all $t\geq 0$
	\begin{equation}
		\mathbb{E}[ X_t ] 
		\leq
		 K_b \int_0^t \mathbb{E}[ X_s ] \vd s.		
	\end{equation} 
	Since the coefficients are bounded,  $ \mathbb{E}[ X_t ] <\infty$ for all $t>0$, and Gronwall lemma yields $\mathbb{E}[ X_t ]  = 0$, which implies that $\mathbb{P}$-a.s., $X_t =0$. 

	If $b_2$ is Lipschitz, then we use that $b_2 \leq b_1$ to obtain 
	\begin{equation}
		 b_2(X^{(2)}) - b_1(X^{(1)}) 
		\leq   b_2(X^{(2)}_s) - b_2(X^{(1)}_s).
	\end{equation}
	
	The assumption of boundedness of the coefficients can be removed via a localization argument. 
	The proof is thus completed.  

\subsubsection{When $\nu\not\equiv 0$.}
    Let $I^{(i)}$ denote the state space of $X^{(i)}$. If $I^{(1)}$ and $I^{(2)}$ do not intersect, then the statement is trivial. 
    Denote by $x^{*}$ a point in the intersection of the state spaces and consider the transformations in Section~\ref{sec:proof:PU} with $I=I^{(1)}\cup I^{(2)}$.  
    Let us now consider $Y^{(i)}=\Phi_\nu(X^{(i)})$.
    By Lemma~\ref{lem:nu_coeff_new} 
    and Lemma~\ref{lem:nu_removal}, $Y^{(i)}$ is a weak non-explosive solution to~\eqref{eq:SDE_no_nu} with coefficients $\widetilde b_i = b_i\circ \Psi_{\nu} \cdot g_{\nu} \circ \Psi_{\nu}$ and $\widetilde \sigma = \sigma \circ \Psi_{\nu} \cdot g_{\nu} \circ \Psi_{\nu}$ satisfying hypothesis~\hypref{hyp:engel-schm:modif}.
    The assumptions of our comparison results in the case $\nu\equiv 0$ are satisfied, hence $\mathbb{P}$-a.s.~for all $t>0$,  $Y^{(1)}_t \geq Y^{(2)}_t$. Since $\Psi_\nu$ is strictly increasing, $\mathbb{P}$-a.s.~for all $t>0$, $X^{(1)}_t=\Psi_{\nu}(Y^{(1)}_t) \geq \Psi_{\nu}(Y^{(2)}_t) = X^{(2)}_t$.

\section{Supplementary results on T-CKLS}
\label{sec:app:CKLS}

In this appendix, considering the T-CKLS as a toy example, we illustrate the
standard method for weak existence and uniqueness based on a reduction to a driftless SDE described in Appendix~\ref{sec:preliminary} and also
Theorem~\ref{Theorem_uniqueness_thresh}.
This provides an alternative route to the pasting technique exploited in
Section~\ref{sec:application} for weak and strong existence and uniqueness. 
In addition, we provide a construction via time change of (reflected) Brownian motion. 

We begin by recalling the definition of the T-CKLS process.
Let $x_0>0$, $\theta>0$, the T-CKLS process $X$ is defined as the solution of the following SDE:
\begin{equation}
X_t = x_0 + \displaystyle \int_0^t \bigl(a(X_s) - b(X_s) X_s\bigr) \mathrm{d}s
    +  \displaystyle \int_0^t \sigma(X_s)\mathrm{d}B_s.
\label{eq:T-CKLS_app}
\end{equation}

The drift and diffusion coefficients are given by
\begin{equation}
a(x)-b(x)x 
=
\begin{cases}
a_+ - b_+ x, & \text{if } x \ge \theta,\\[4pt]
a_- - b_- x, & \text{if } x < \theta,
\end{cases}
\qquad
\sigma(x)
=
\begin{cases}
\sigma_+ x^{\gamma_+}, & \text{if } x \ge \theta,\\[4pt]
\sigma_- x^{\gamma_-}, & \text{if } x < \theta,
\end{cases}
\end{equation}
where $a_-, \sigma_-, \sigma_+$ are strictly positive constants,
$\gamma_- \in [1/2,1]$, and $\gamma_+ \in [0,1]$.

Once weak existence, uniqueness in law, and non-explosion have been established
for \eqref{eq:T-CKLS_app}, strong existence and uniqueness follow by
applying either Theorem~\ref{thm:pasting} or
Theorem~\ref{Theorem_uniqueness_thresh}.
While the former is used in the main body of the paper, we illustrate here the latter by verifying Hypothesis~\hypref{hyp:pathwise} (weak existence, uniqueness in law, and non-explosion) and Hypothesis~\hypref{hyp:cond:bsigma} (and therefore~\hypref{hyp:engel-schm:modif}). Non-explosion and Hypothesis~\hypref{hyp:cond:bsigma} are addressed in Section~\ref{sec:C1}, while weak existence and uniqueness are treated in Section~\ref{sec:alter_weak_ckls}. In alternative to the proof provided in Section~\ref{sec:alter_weak_ckls}, Weak existence, uniqueness in law, and non-explosion can be established via Proposition~\ref{prop:weak_pasting} as in Section~\ref{sec:application}. 
In what follows, we also prove that the solution to \eqref{eq:sticky_skew_CKLS} is non-negative (and in some cases strictly positive), that it is a strong Markov process, and we provide a construction via a time change of (reflected) Brownian motion (see Remark~\ref{Remark:C7}).

\subsection{Preliminary results on the T-CKLS model}\label{sec:C1}

Before analyzing the weak well-posedness, we establish a basic structural property
satisfied by any weak solution to~\eqref{eq:T-CKLS_app}.
We show that the process does not explode and cannot accumulate time at the boundary point~$0$,
and that its right local time at~$0$ vanishes. 
This fact plays a crucial role because it is related to the \emph{fundamental solution} of the associated driftless SDEs (see Appendix~\ref{sec:preliminary}).

\begin{prop}[Zero occupation and vanishing local time at $0$]
\label{prop:zero_time}
Let $X$ be any (weak) solution to the SDE~\eqref{eq:T-CKLS_app}. Then, for all $t>0$,
almost surely,
\begin{equation}
\int_0^t \mathds{1}_{\{0\}}(X_s)\,\mathrm{d}s = 0
\qquad\text{and}\qquad
\ell_t^{0}(X) = 0.
\label{eq:no_time}
\end{equation}
\end{prop}

\begin{proof}
We prove that any solution to~\eqref{eq:T-CKLS_app} necessarily satisfies
\eqref{eq:no_time}.

By~\cite[Theorem~VI.1.7]{revuz2013continuous}, we have almost surely
\begin{equation}
\ell_t^{0}(X) - L_t^{0^-}(X)
    = 2 a_- \int_0^t \mathds{1}_{\{0\}}(X_s)\,\mathrm{d}s,
\label{eq:local_occup}
\end{equation}
where $L_t^{0^-}(X)$ denotes the left local time,
$
L_t^{0^-}(X) = \lim_{h \to 0,\, h<0} \ell_t^{h}(X),
$ 
which is non-negative.

Using the occupation times formula (see
\cite[Corollary~VI.1.6]{revuz2013continuous}), for any fixed
$\varepsilon\in(0,\theta)$ we obtain almost surely
\begin{align}
t
&\ge \int_0^t \mathds{1}_{(0,\varepsilon)}(X_s)\,\mathrm{d}s 
= \int_0^t
   \frac{\mathds{1}_{(0,\varepsilon)}(X_s)}
        {\sigma_-^2 (X_s)^{2\gamma_-}}\,\mathrm{d}\langle X,X\rangle_s
 = \int_0^\varepsilon
        \frac{\ell_t^y(X)}{\sigma_-^2 y^{2\gamma_-}}\,\mathrm{d}y.
\end{align}

Since the integral on the right-hand side is finite, it follows that
$\ell_t^0(X)=0$ almost surely.
Plugging this into~\eqref{eq:local_occup} yields
\[
    -L_t^{0^-}(X)
    = 2 a_- \int_0^t \mathds{1}_{\{0\}}(X_s)\,\mathrm{d}s.
\]
The left-hand side is non-positive whereas the right-hand side is non-negative.
Thus both sides must vanish, proving~\eqref{eq:no_time}.
\end{proof}

The following non-explosion result is a direct application of Proposition~\ref{prop:no_explosion_linear}.

\begin{prop}
\label{prop:non_explosion}
Let $(X,B)$ be a (weak) solution to~\eqref{eq:T-CKLS_app}.
Then, almost surely, $X$ does not explode in finite time.
\end{prop}

\begin{lemma}
Hypothesis~\hypref{hyp:cond:bsigma} are satisfied by the coefficients~\eqref{eq:drift_TCKLS}.
\end{lemma}
\begin{proof}
    Condition~\ref{item_classic_cont:th_uniqueness} holds with $\rho(x)=\sigma_- |x|^{2\gamma_-}$ and $\rho(x):=\sigma_+ |x|^{2\gamma_+}$. Condition~\ref{item_disc:th_uniqueness} is also readily verified, since the volatility is continuous and increasing on each side of the threshold, which is a sufficient condition to ensure the required property. 
\end{proof}

\subsection{Weak existence and uniqueness for the T-CKLS process}
\label{sec:alter_weak_ckls}

The argument relies on the following simple idea already sketched in Appendix \ref{sec:preliminary}. By Proposition~\ref{prop:non_explosion} and Proposition~\ref{prop:zero_time} we know that if there exists a weak solution to the T-CKLS SDE, $(X,B)$, then it is non-explosive and it spends no time at $0$, i.e.~it satisfies~\eqref{eq:no_time}. 
We consider a suitable space transformation $Y=S(X)$ which solves a driftless SDE~\eqref{eq:driftless_after_scale} with respect to the same Brownian motion with a suitable diffusion coefficient $\widehat \sigma$ depending on $b,\sigma$. 
Since $X$ does not spend any time at any point, $Y$ should be a \emph{fundamental solution} in the application of the results recalled in Appendix~\ref{sec:preliminary} 
to the driftless SDE~\eqref{eq:driftless_after_scale}.
This can be subtle because of the shape of $\widehat \sigma$, but the a-priori information on $X$, based on Proposition~\ref{prop:zero_time} and Proposition~\ref{prop:non_explosion}, helps. We illustrate this in the following paragraphs.

\paragraph{Reducing to a driftless SDE.}
For T-CKLS, the \emph{scale function} $S$ in~\eqref{eq:scale_def} has $S(x) =\int_\theta^x s(y) \vd y$ with
\begin{equation}
    s(x)=s_-(x):=\theta^{2a_-/\sigma_-^2} \exp\left( \frac{-2 b_-  \theta}{\sigma_-^2}\right) x^{-2a_-/\sigma_-^2} \exp\left( \frac{2 b_-  x}{\sigma_-^2}\right), \quad x\in (0, \theta),
\label{eq:scale_cir_0_theta}
\end{equation}
if $\gamma_-=1/2$, and if $\gamma_-\in (1/2,1)$ there exists a constant $C\in (0,+\infty)$ (depending on $\theta$) such that
\begin{equation}
    s(x)=C \exp\left( \frac{2a_-}{\sigma_-^2} \frac{x^{1-2\gamma_-}}{2\gamma_--1}  + \frac{b_-x^{2(1-\gamma_-)}}{\sigma_-^2 (1-\gamma_-)}\right), \quad x\in (0, \theta)
\end{equation}
and if $\gamma_-=1$ there exists a constant
$C\in (0,+\infty)$ (depending on $\theta$) such that 
\begin{equation}
    s(x)=C \exp\left( \frac{2a_-}{\sigma_-^2} \frac1x \right) x^{2{b_-}/{\sigma_-^2}} \quad x\in (0, \theta).
\end{equation}
Note that $S$ is strictly increasing. 
The behavior below $\theta$ depends only on the value of the coefficients below $\theta$ and the value above $\theta$, depends on the coefficients above the threshold $\theta$.  

Let us first consider the behavior above the threshold. 
$S_{+\infty} :=\lim_{x\to + \infty} S(x)$ exists in $(0,+\infty]$ and 
$S_{+\infty}=+\infty$ in the following cases
if $\gamma_+=1/2$ and either $b_+>0$ or [$b_+=0$ and $2 a_+\leq \sigma_+^2$] 
or $\gamma_+\in [0,1/2)$ and $b_+>0$ or [$b_+=0$ and $a_+\leq 0$],
or $\gamma_+\in (1/2,1)$ and $b_+\geq 0$,
or $\gamma_+=1$ and $ 2 b_+\geq - \sigma^2_+$. 
Otherwise, we have $S_{+\infty}<\infty$, that is 
        if $\gamma_+=1/2$ and either $b_+<0$ or [$b_+=0$ and $2 a_+> \sigma_+^2$] 
        or $\gamma_+\in [0,1/2)$ and $b_+<0$ or [$b_+=0$ and $a_+ > 0$],
        or $\gamma_+\in (1/2,1)$ and $b_+< 0$,
        or $\gamma_+=1$ and $ 2 b_+ < - \sigma^2_+$.

Let us now consider the behavior below the threshold. 
$S_0:= \lim_{x\to 0^+} S(x) = -\infty$ either if $\gamma_->1/2$ or if $\gamma_-=1/2$ and $\sigma_-^2 \leq 2 a_-$. 
In this case $S$ is a bijective function from $(0,+\infty)$ to $(-\infty,S_{+\infty}) \subseteq \mathbb{R}$. 
If $\gamma_-=1/2$ and $\sigma_-^2 > 2 a_-$, then $S_0:=\lim_{x\to 0^+} S(x)$ is finite (strictly negative), but $\lim_{x\to 0^-} S(x) =-\infty$. 
$S|_{[0,+\infty)}$ is a continuous, strictly increasing bijective function from $[0,+\infty)$ to $[S_0,S_{+\infty})$. %
In both cases, let $R$ denote the inverse of $S$. 
Note that $R$ is a strictly increasing and continuous function. 
Summarizing, $S$ is a strictly increasing continuous bijective function
\begin{enumerate}
    [label=\textbf{\upshape S\arabic*}, nolistsep]
    \item \label{item:case:n0pi} 
        from $(0,+\infty)$ to $\mathbb R$ 
        when $S_0=-\infty$ 
        and $S_{+\infty}=+\infty$.  
    \item \label{item:case:n0pf} 
    from $(0,+\infty)$ to $(-\infty, S_{+\infty})$ when $S_{+\infty} \in (0,+\infty)$ and $S_0=-\infty$.  
    \item \label{item:case:nipi} 
        from $[0,+\infty)$ to $[S_0, +\infty)$  
        when $S_0$ is finite and $S_{+\infty}=+\infty$. 
    \item \label{item:case:nipf}
        from $[0,+\infty)$ to $[S_0, S_{+\infty})$
        when $S_0$ and $S_{-\infty}$ are finite. 
\end{enumerate}

We now determine the diffusion coefficient $\widehat \sigma$ of the driftless SDE in~\eqref{eq:driftless_after_scale}. Assume that there exists a weak solution to~\eqref{eq:T-CKLS_app} which is non-negative, and let $Y := S(X)$, then $Y$ is solution to
\begin{equation}
    Y_t=\int_0^t \sigma(R(Y_t)) {|R(Y_t)|^{\gamma_r(R(Y_t))}} s\circ R(Y_t)dB_t.
    \label{Y_scale}
\end{equation}
Note that, although the introduction of $Y$ as a fundamental solution to~\eqref{Y_scale} has been done via a space transformation $S$ of $X$, nothing obliges that any (fundamental) solution to~\eqref{Y_scale} has state space contained in $[S_0,S_{+\infty})$. 
The criteria in Appendix~\ref{sec:preliminary}, are based on the sets $N_\sigma$ and $E_\sigma$ of Section \ref{sec:preliminary}.
\begin{equation}
    N_\sigma = \{ x\in \mathbb{R} \colon |R(x)|^{\gamma_r(R(x))} s \circ R(x)=0 \}
\end{equation}
and, with o.s.c.~meaning {\it open set containing}, 
\begin{equation}
    \begin{split}
    E_\sigma 
    & = \left\{ x\in \mathbb{R} \colon \int_{R(\mathcal{U}(x))} \frac1{
    |y|^{2\gamma_r(y)} s(y)} \vd y =+\infty \ \forall \mathcal{U}(x) \text{ o.s.c. } x\right\}.
    \end{split}
\end{equation} 
We now show that 
$E_{\sigma}= \{S_{+\infty}\}\cap \mathbb{R} \subseteq N_\sigma \subseteq \{S_0,S_{+\infty}\}\cap \mathbb{R}$, so in the ``worst case scenario'' we are in case~\ref{item:es:diff0}, hence there exists a unique (in law) fundamental solution to~\eqref{Y_scale}.

\paragraph{Weak existence.}

We deal with existence by providing a non-negative weak solution. 
We have just justified existence of a unique (in law) fundamental solution to~\eqref{Y_scale}, say $(Y,B)$. 
We now justify that this solution is in the domain of $R$. Then one obtains a solution to the T-CKLS $(X:=R(Y),B)$ SDE~\eqref{eq:T-CKLS_app}, by It\^o-Tanaka formula. 

We now see that in the two cases~\ref{item:case:n0pi} and \ref{item:case:n0pf} the solution to~\eqref{Y_scale} is unique and it is fundamental. In these cases, we obtain a positive weak solution to~\eqref{eq:T-CKLS_app}.  
In the two cases~\ref{item:case:nipi} and \ref{item:case:nipf} we now show that there is a unique (in law) fundamental solution and we construct such a solution $(Y,B)$. We can then prove that $Y$ takes values in the domain of $R$, thanks to the a-priori information on the solutions to~\eqref{eq:T-CKLS_app}. So $X_\cdot:=R(Y_\cdot)$ yields  a non-negative solution to~\eqref{eq:T-CKLS_app}. 

{\it Consider \ref{item:case:n0pi}}, then $N_\sigma=E_\sigma=\emptyset$, that corresponds to~\ref{item:es:0}.  
Let us consider $Y$ the unique (fundamental) weak solution to~\eqref{Y_scale}. 
Thus, existence of weak solutions to~\eqref{eq:T-CKLS_app} is ensured. Moreover since $R$ takes values in $(0,+\infty)$, $R(Y)$ is strictly positive.

{\it Consider \ref{item:case:n0pf}}. 
In this case $N_\sigma=E_\sigma=\{S_{+\infty}\}$, which corresponds to~\ref{item:es:equal}. 
Hence, weak existence and uniqueness in law hold for~\eqref{Y_scale}. 
Note that $\tau:=\inf\{t>0 \colon Y_t \geq S_{+\infty}\} = \inf\{t >0 \colon R(Y_t) = +\infty\}$. 
For all $t\geq 0$, $R(Y)$ satisfies~\eqref{eq:T-CKLS_app} on the event $\{t<\tau\}$. 
Proposition~\ref{prop:non_explosion} ensures that $\mathbb{P}(\tau<\infty)=0$. 
Hence, we just proved existence of solutions to~\eqref{eq:T-CKLS_app}. 
Moreover, $R(Y)$ is strictly positive and $S(R(Y))=Y$ (since we are in case~\ref{item:es:equal}) is (the unique) solution to~\eqref{Y_scale} and the state space of $Y$ is contained in $(-\infty, S_{+\infty})$. 

\begin{remark}
As a step in the above proof, we checked that $S_{+\infty}$ does not belong to the state space of $Y$ by using non-explosion of solutions to~\eqref{eq:T-CKLS_app}. 
Alternatively, one can show that $S_{+\infty}$ is a non-exit point for $Y$, e.g., via Feller boundary classification in order to prove that $S_{+\infty}$ is not in the state space.
\end{remark}

{\it Consider \ref{item:case:nipi}}. 
Note that 
$N_\sigma=\{S_0\}$ if $4 a_- <\sigma^2_-$, otherwise $N_\sigma=\emptyset$
and 
\begin{equation}
   E_\sigma=\{x\in \mathbb{R} \colon  \forall \varepsilon>0, \ \int_{(x-\varepsilon, x+\varepsilon)} {
     |y|^{-(1-2 a_-/\sigma_-^2)}} \vd y =+\infty\}=\emptyset.
\end{equation}

Then, $\emptyset=E_\sigma \subseteq N_\sigma$, that is either case~\ref{item:es:0} or ~\ref{item:es:diff}.

{\it Consider \ref{item:case:nipf}}. 
Note that $E_\sigma=\{S_{+\infty}\}$ and either $N_\sigma=\{S_0,S_{+\infty}\}$ if $4 a_- <\sigma^2_-$, or  $N_\sigma=\{S_{+\infty}\}$. If $4 a_- <\sigma^2_-$ we are in case~\ref{item:es:diff0}, otherwise~\ref{item:es:equal}.

\begin{remark}
    In the latter two cases, we cannot just apply $R$ to the unique fundamental solution to~\eqref{Y_scale}, $(Y,B)$, because we have no knowledge of its state space. We could only establish weak existence and uniqueness for equation~\eqref{eq:T-CKLS_app} up to the first hitting time of 0. To overcome this limitation we construct $Y$ as a transformation of a time-changed (reflected) Brownian motion. 
\end{remark}

\paragraph{Construction of the fundamental solution to~\eqref{Y_scale} in cases~\ref{item:case:nipi} and~\ref{item:case:nipf} as time change of a reflected Brownian motion.}
Let $Z$ be a standard Brownian motion and let $W:=|Z+ S(X_0)-S_0|$ a reflected Brownian motion starting at $S(X_0)-S_0$. Note that the latter quantity is positive since $S$ is increasing. By Tanaka formula 
$W$ satisfies 
$W_t = S(X_0)-S_0 + \int_0^t \sgn(Z_s+S(X_0)-S_0) \vd Z_s + L^{S_0-S(X_0)}_t(Z)$.
Note that $\hat{B}_t:=\int_0^t \sgn(Z_s+ S(X_0)-S_0) \vd Z_s$ is a standard Brownian motion. 
Hence $W_t= S(X_0)-S_0 + \hat{B}_t + L^{S_0-S(X_0)}_t(Z)$.
Let $(\mathcal{F}_t)_t$ be the completed natural filtration associated to $W$. 

Recall that in case~\ref{item:case:nipi} $S_{+\infty}=+\infty$ and in case~\ref{item:case:nipf} $S_{+\infty}<+\infty$. 
Let $\tau:=\inf\{s\geq 0 \colon W_s=S_{+\infty}-S_0\}$ and, for $t<\tau$, let  
\begin{equation}
    A_t:=\int_0^t \ind{[S_0,S_{+\infty}]}(W_s+S_0) \rho(R(W_s +S_0)) \vd s
    =\int_0^t \rho(R(W_s +S_0)) \vd s,
\end{equation}
where 
$\rho = (\sigma s)^{-2}$. 

\emph{We show that $A_t \ind{[0,\tau)}(t)<+\infty$ is well defined and $A_t$ is $\mathbb{P}$-a.s.~finite for every $t\in [0,\tau)$.}
Note that $R$ is defined on $[S_0,S_{+\infty})$, and $\tau$ ensures that for $t<\tau$, the integrand is well defined and $A_t$ is well defined on $[0,+\infty]$ because the integrand is non-negative. 
In the next lines we show that $\mathbb{P}$-a.s.~$A_t \ind{[0,\tau)}(t)<+\infty$.
By applying the occupation times formula, we obtain
\begin{align}
    A_t\mathds{1}_{[0,\tau)}(t)&\leq \ind{[0,\tau)}(t)\int_0^t \rho(R(W_s +S_0)) \,\mathrm{d}s \\
    &\leq  \ind{[0,\tau)}(t) \int_0^{S_{+\infty}-S_0} \rho(R(x+S_0))L_t^x(W)\,\mathrm{d}x\\
    &\leq \ind{[0,\tau)}(t) \int_{0}^{+\infty} \frac{L_t^{S(x)-S_0}(W)}{\sigma^2(x)s(x)}\,\mathrm{d}x
    \\
    & = \int_{0}^{\theta} \frac{L_t^{S(x)-S_0}(W)}{\sigma^2(x)s(x)}\,\mathrm{d}x +  \ind{[0,\tau)}(t) \int_{\theta}^{+\infty} \frac{L_t^{S(x)-S_0}(W)}{\sigma^2(x)s(x)}\,\mathrm{d}x.
\end{align}
 Using~\eqref{eq:scale_cir_0_theta}, together with assumptions~\ref{item:case:nipi}-\ref{item:case:nipf} we can show that the first summand on the right-hand-side of the latter equation is $\mathbb{P}$-a.s.~finite.  
 Let us focus on the second summand. Note that for every $t\in [0,\tau)$, we have $\max_{s\in [0,t]}W_s+S_0 \in [S_0,S_{+\infty})$ and so $\bar x:=R\left(\max_{s\in [0,t]}W_s+S_0\right)\in [0,+\infty)$. Therefore, $L_t^{S(x)-S_0}(W)= 0$ for all $x>\bar x$. If $\bar x \leq \theta$ we are done, otherwise, we are reduced to check if 
 $\int_{\theta}^{\bar x} L^{S(x)-S_0}(W)/(\sigma^2(x)s(x)) \vd x$ is $\mathbb{P}$-a.s.~finite which is the case because $\sigma$ is bounded on $[\theta,\bar x]$, $s$ is continuous on $[\theta,\bar x]$ and 
\begin{equation}
    \int_r^{+\infty} L_t^{S(x)-S_0}(W) \vd x \leq \int_{0}^{+\infty} L_t^{y}(W) \vd y = t <\infty
\end{equation}
(the last equality follows from the occupation times formula). 

\emph{We introduce the time-change associated to $A$.} 
Note that $(A_t)_{t\in [0,\tau)}$ is continuous and strictly increasing, hence we can define $A_{\tau}:=\lim_{s\to \tau-} A_s\in [0,+\infty]$. 
$A_\tau$ is $(\mathcal{F}_t)_{t\geq 0}$-adapted. 
Its right inverse $\tau_t:=\inf\{s \geq 0 \colon A_s >t\}$ is strictly increasing and continuous as a (random) function from $[0,A_{\tau})$ to $[0,\tau)$, hence it is a time-change, see, e.g., \cite[Definition V.1.2]{revuz2013continuous}. 
Moreover, $\tau_{A_{\tau}^-}=\inf\{s>0 \colon A_s \geq A_\tau\}=\tau$ 
    and for all $t\in [0,\tau)$ we have $\tau_{A_t}=t$ and, if $t<A_{\tau}$, then there exists $s \in [0,A_\tau)$ such that $t<A_s$, so $\tau_t <\tau_{A_s}=s < A_\tau$, i.e.~$(\tau_t)_{t\in [0,A_{\tau})}$ is $\mathbb{P}$-a.s.~finite time-change.

In conclusion $(\tau_t)_{t\in [0,A_{\tau})}$ is a strictly increasing and continuous $\mathbb{P}$-a.s.~finite time-change taking values in $[0, \tau)$.  
Note that, up to considering the time-change $\tilde\tau_t:=\tau_{t\wedge A_\tau}$, we can consider $(\tau_t)_{t\geq 0}$ taking values in $[0,\tau]$, instead of $(\tau_t)_{t\in [0,A_\tau)}$. 

\emph{We are now going to show that the process defined by $\tilde{Y}_t:= W_{\tau_t}+S_0$ for $t< A_\tau$, solves~\eqref{Y_scale} and that $A_\tau=+\infty$ while constructing a solution to~\eqref{eq:T-CKLS_app}.}

By~\cite[Exercice VI.1.17]{revuz2013continuous} it holds $\mathbb{P}$-a.s.~that $L_t^0(W)=2L_t^{S_0-S(X_0)}(Z)$ for all $t \in (0,+\infty)$. 
By~\cite[Exercice VI.1.27]{revuz2013continuous}, $ L_t^{S_0}(\tilde{Y}) =L_{\tau_t}^0(W)$ $\mathbb{P}$-a.s.~for all $t\geq 0$.

Let us show that $\mathbb{P}$-a.s., for $t\geq 0$, we have 
$\hat{B}_{\tau_t} = \int_0^t  \sigma(R(\tilde{Y}_s)) s(R(\tilde{Y}_s)) \vd \hat{B}_s$. 
To do so, note that for all $t\geq 0$ $\hat{B}_t= \int_0^t \vd \hat{B}_s$ and so 
$
    \hat{B}_{\tau_t}=\int_0^{\tau_t}  \vd \hat{B}_s
$. 
This is an $(\mathcal{F}^{\hat{B}}_{\tau_{t}})_t$-local martingale \cite[Proposition V.1.5.i)]{revuz2013continuous} 
with quadratic variation 
\begin{equation}
    \tau_{t} = \int_0^{\tau_{t}} \vd s 
    = \int_0^{t\wedge A_\tau} \tau_s' \vd s 
    = \int_0^{t\wedge A_\tau} \frac1{A'_{\tau_s}} \vd s 
    = \int_0^t \ind{[0,A_\tau)}(s) \rho^{-1}(R(W_{\tau_s}+S_0)) \vd s.
\end{equation}
The martingale representation theorem (e.g.,~\cite[Proposition V.3.8]{revuz2013continuous} ensures that there exists an $(\mathcal{F}^{\hat{B}}_{\tau_t})_{t<A_\tau}$ Brownian motion $B$ such that 
\begin{equation}
    \hat{B}_{\tau_t} 
        = \int_0^t \rho^{-1/2}(R(W_{\tau_s}+ S_0))\vd B_s
\quad \text{hence} \quad 
    \hat{B}_{\tau_t} =\int_0^t  \sigma(R(\tilde{Y}_s)) s(R(\tilde{Y}_s)) \vd B_s.
\end{equation}
Therefore, we have that $\mathbb{P}$-a.s.~for all $t<A_\tau$,  $\tilde{Y}_t=W_{\tau_t} +S_0$ and
\begin{equation}
    \tilde{Y}_t = S(X_0) + \int_0^t  \sigma(R(\tilde{Y}_s)) s(R(\tilde{Y}_s)) \vd B_s + \frac12 L^{S_0}_{t}(Y).
\end{equation} 
Let $X_t:=R(\tilde{Y}_t)=R(W_{\tau_t} +S_0)$ for $t<A_\tau$ and by It\^o formula we get that $\mathbb{P}$-a.s.~for $t<A_\tau$
\begin{align}
    X_t 
    & = X_0
    + \int_0^t  R'(\tilde{Y}_t) \sigma(R(\tilde{Y}_s)) s(R(\tilde{Y}_s)) \vd B_s 
    + R'(S_0) \frac12 L^{S_0}_{t}(\tilde{Y}) \nonumber
    \\
    & \qquad +\frac12 \int_0^t R''(\tilde{Y}_t) \sigma^2(R(\tilde{Y}_s)) s^2(R(\tilde{Y}_s)) \vd s,
    \label{eq:X_transfo}
\end{align}
Since $\lim_{x\to 0^+}  R'(S(x))= \lim_{x\to 0^+}  1/{s(x)} =0$, \eqref{eq:X_transfo} is \eqref{eq:T-CKLS_app}. It remains to prove that $A_\tau=+\infty$ $\mathbb{P}$-a.s.
Note that for all $t\in [0,A_{\tau})$ we have $X_{A_t} =R(\tilde{Y}_{A_t})=R(W_{\tau_{A_t}} +S_0)= R(W_t +S_0)$ and $X_{A_\tau} 
    = \lim_{t\to \tau-} R(W_t +S_0)
    = R_{S_{+\infty}}=+\infty$.
Then $A_\tau=+\infty$ $\mathbb{P}$-a.s.~otherwise this would contradict Proposition~\ref{prop:non_explosion}.
Existence of a weak solution to~\eqref{eq:T-CKLS_app} is then proven. 
Note that $X$ is non-negative by construction.
By Proposition~\ref{prop:zero_time}, we know that $\mathbb{P}$-a.s.~$L^0_t(X)=0$ and~\eqref{eq:no_time} holds. 

Hence $Y_t=S(X_t)= S\circ R(\tilde{Y}_t)=\tilde{Y}_t$ solves~\eqref{Y_scale}, takes values in $[S_0, S_{+\infty})$ and it is a fundamental solution.

\begin{remark}
    In our proof of the case~\ref{item:case:nipf}, since we constructed the solution, we do not need to proceed similarly to the case~\ref{item:case:n0pf} as it concerns dealing with $S_{+\infty}$. 
\end{remark}

\paragraph{Uniqueness in law.}
We now discuss uniqueness in law. 
Assume that there are two solutions to~\eqref{eq:T-CKLS_app}, $(X^{(1)},B^{(1)})$ and $(X^{(2)},B^{(2)})$ on two different probability spaces. 
By Proposition~\ref{prop:zero_time}, they $\mathbb{P}$-a.s.~spend 0 time at 0. 
Let $\tau_i$ be the first time $X^{(i)}$ exits from $[0,+\infty)$. 
On the two different probability spaces, $(S(X^{(i)}_s))_{s<\tau_i}$ are fundamental solutions to \eqref{Y_scale} and $\tau_i=\infty$ since the fundamental solution to~\eqref{Y_scale} never goes below $S_0$ 
(in cases~\ref{item:case:n0pi}-\ref{item:case:n0pf} $S_0=-\infty$). 
Therefore the two solutions to~\eqref{eq:T-CKLS_app} are such that $S(X^{(1)})$ and $S(X^{2})$ are equal in law, by uniqueness in law to~\eqref{Y_scale} taking values in the domain of definition of $R$. 
Thus, $X^{(1)} \equiv R\circ S(X^{(1)})$ and $X^{(2)}\equiv R\circ S(X^{2})$ are equal in law as well.

\begin{remark}[Strong Markov Property]
The unique fundamental solution $Y$ to~\eqref{Y_scale} enjoys the strong Markov property (see Section~\ref{sec:preliminary}, Remark~\ref{rem:smarkov_general}). 
The space transformation of $Y$, $X$ solution to~\eqref{eq:T-CKLS_app}, satisfies it as well.
By uniqueness in law, all weak solutions to \eqref{eq:T-CKLS_app} share this property. 
In the cases~\ref{item:case:nipi}-\ref{item:case:nipf}, the solution $X$ to~\eqref{eq:T-CKLS_app} has been constructed by transformation of time-change of reflected Brownian motion (and so of a Brownian motion). 
In the cases~\ref{item:case:n0pi}-\ref{item:case:n0pf}, we can do a similar construction as time-change of a Brownian motion or we can deduce it by the results recalled in Remark~\ref{rem:smarkov_general}, after having proved weak existence and uniqueness (in law). 
\label{Remark:C7}
\end{remark}

\begin{remark}[Non negativity of the solution]
In cases~\ref{item:case:nipi}-\ref{item:case:nipf} we have constructed a non-negative solution and in the cases~\ref{item:case:n0pi}-\ref{item:case:n0pf} the solution is strictly positive. 
By uniqueness in law, all weak solutions to \eqref{eq:T-CKLS_app} share this property.
\end{remark}

\paragraph{Acknowledgments.} 
The author Sara Mazzonetto would like to thank the Isaac Newton Institute for Mathematical Sciences, Cambridge, for support and hospitality during the programme \emph{Stochastic systems for anomalous diffusion} where work on this paper was undertaken. This work was supported by EPSRC grant no EP/Z000580/1.


\end{document}